\documentclass[11pt]{article}
\usepackage{geometry}
\usepackage{graphicx}
\usepackage{indentfirst} 
\usepackage{caption}
\usepackage{booktabs}
\usepackage{framed}
\usepackage{subfigure}
\usepackage{xcolor}
\usepackage{amsfonts}
\usepackage{mathrsfs}
\usepackage{amsmath}
\usepackage{amsthm}
\usepackage{amssymb}
\def\ds{\displaystyle}
\def\={\buildrel \triangle \over =}
\def\a{\alpha}

\def\d{\delta}
\def\e{\varepsilon}

\def\l{\lambda}

\def\f{\varphi}

\def\ns{\noalign{\ss} }

\def\D{\Delta}

\def\ms{\medskip}

\def\q{\quad}
\def\qq{\qquad}

\def\dbF{{\mathbb{F}}}

\def\dbM{{\mathbb{M}}}

\def\dbP{{\mathbb{P}}}

\def\mE{{\mathbb{E}}}

\def\3n{\negthinspace \negthinspace \negthinspace }
\def\2n{\negthinspace \negthinspace }
\def\1n{\negthinspace }

\def\Ma{\textbf{Mass}}
\def\St{\textbf{Stiff}}
\def\Md{\textbf{MD}}
\def\Mg{\textbf{MG}}
\def\Au{\textbf{A}}
\def\Ay{\textbf{B}}

\def\cA{{\cal A}}
\def\cB{{\cal B}}

\def\cF{{\cal F}}

\def\pa{\partial}

\def\cd{\cdot}
\def\cds{\cdots}

\def\div{\hbox{\rm div$\,$}}

\def\|{||}
\def\({\Big (}
\def\){\Big )}
\def\[{\Big[}
\def\]{\Big]}
\def\be{\begin{equation*}}
\def\bel{\begin{equation}\label}
\def\ee{\end{equation}}
\def\bt{\begin{theorem}}
\def\bcd{\begin{condition}}
\def\ecd{\end{condition}}
\def\et{\end{theorem}}
\def\bc{\begin{corollary}}
\def\ec{\end{corollary}}
\def\bde{\begin{definition}}
\def\ede{\end{definition}}
\def\bl{\begin{lemma}}
\def\el{\end{lemma}}
\def\bp{\begin{proposition}}
\def\ep{\end{proposition}}
\def\br{\begin{remark}}
\def\er{\end{remark}}
\def\ba{\begin{array}}
\def\ea{\end{array}}
\def\ed{\end{document}}
\def\ns{\noalign{\ms}}
\def\ds{\displaystyle}

\def\Om{\Omega}

\newtheorem{lemma}{Lemma}[section]
\newtheorem{remark}{Remark}[section]
\newtheorem{example}{Example}[section]
\newtheorem{theorem}{Theorem}[section]
\newtheorem{corollary}{Corollary}[section]

\newtheorem{definition}{Definition}[section]
\newtheorem{proposition}{Proposition}[section]
\newtheorem{condition}{Condition}[section]
\numberwithin{equation}{section}
\allowdisplaybreaks

\title{\bf Determination the Solution of a Stochastic Parabolic Equation by the Terminal Value}
\author{Fangfang Dou\thanks{Corresponding author. School of Mathematical Sciences, University of Electronic Science and Technology of China, Chengdu, China. Email: fangfdou@uestc.edu.cn.} \quad and \quad Wanli Du\thanks{School of Mathematical Sciences, University of Electronic Science and Technology of China, Chengdu, China.}}
\date{}                       

\begin{document}
\maketitle
\begin{abstract}
This paper studies the inverse problem of determination the history for a stochastic diffusion process, by means of the value  at
the final time $T$. By establishing a new Carleman estimate, the conditional stability of the problem is proven. Based on the idea of Tikhonov method, a regularized solution is proposed. The analysis of the existence and uniqueness of the regularized solution, and proof for error estimate under an a-proior assumption are present. Numerical verification of the regularization, including numerical algorithm and examples are also illustrated.
\end{abstract}

\section{Introduction}

The stochastic parabolic equations are widely used to describe many diffusion processes perturbed by stochastic noises, such as  the evolution of the density of a bacteria population, the propagation of an electric potential in a neuron, etc., (e.g., \cite{dPZ2014,D1972,K2008,LZ2021S}). In this paper, we study an inverse problems for stochastic parabolic equations, i.e., determining the solution  from the terminal measurement. To be more precisely,  we first introduce some notations.

Let $T>0, G \subset \mathbb{R}^{n}(n \in \mathbb{N})$ be a given bounded domain with a $C^{2}$ boundary $\Gamma .$  Put
$
Q \triangleq(0, T) \times G$ and $\Sigma \triangleq(0, T) \times \Gamma$.
 
Let $\left(\Omega, \mathcal{F},\left\{\mathcal{F}_{t}\right\}_{t \geqslant 0}, \dbP\right)$ be a complete filtered probability space on which a one dimensional standard Brownian motion $W(\cd)$ is defined.

Let $H$ be a Banach space. Denote by $L^2_{\cF_t}(\Om;H)$ ($t\geq 0$) the space of all $H$ -valued random variables $\xi$ satisfying $\mE|\xi|_H^2<\infty$; by $L_{\mathbb{F}}^{2}(0, T ; H)$ the space  of all $H$ -valued $\left\{\mathcal{F}_{t}\right\}_{t \geqslant 0}$ -adapted processes $X(\cdot)$ such that $\mathbb{E}(|X(\cdot)|_{L^{2}(0, T ; H)}^{2})<\infty$; by $L_{\mathbb{F}}^{\infty}(0, T ; H)$ the space of all $H$ -valued $\left\{\mathcal{F}_{t}\right\}_{t \geqslant 0}$ -adapted bounded processes; and by $L_{\mathbb{F}}^{2}(\Omega ; C([0, T] ; H))$ the space  of all $H$ -valued $\left\{\mathcal{F}_{t}\right\}_{t \geqslant 0}$-adapted processes $X(\cdot)$ satisfying that $\mathbb{E}(|X(\cdot)|_{C(0, T ; H)}^{2})<\infty $. All these spaces are Banach spaces with the canonical norms (e.g.,\cite[Chapter 2]{LZ2021S}).

Consider the following stochastic parabolic equation:
\begin{equation}\label{p1}
\left\{\begin{array}{ll}\ds\mathrm{d} u-\sum_{i, j=1}^{n}(a^{i j} u_{i})_j \mathrm{d} t=\left( b_{1}\cd\nabla u +b_{2} u+f\right) \mathrm{d} t+\left(b_{3} u+g\right) \mathrm{d} W(t) & \text { in } Q, \\ \ns\ds u=0 & \text { on } \Sigma, \\ \ns\ds u(0)=u_{0}, & \text { in }G,
\end{array}
\right.
\end{equation}
where $u_{0} \in L^{2}_{\mathcal{F}_{0}}\left(\Omega; L^{2}(G)\right)$ and $u_i=\frac{\partial u}{\partial x_i}$ for $i=1,\cds,n$.

Throughout this paper, we make the following assumptions on the coefficients $
a^{i j}: \Omega \times Q \rightarrow \mathbb{R}, i, j=1,2, \ldots, n
$:
\begin{itemize}
\item[(H1)] $a^{i j} \in L_{\mathbb{F}}^{2}\left(\Omega ; C^{1}\left([0, T] ; W^{2, \infty}(G)\right)\right)$ and $a^{i j}=a^{j i}$ for $i, j=1,2, \ldots, n$.

\item[(H2)] There is a constant $\sigma>0$ such that
$$
\sum_{i, j=1}^{n} a^{i j}(\omega, t, x) \xi^{i} \xi^{j} \geqslant \sigma|\xi|^{2}, \quad(\omega, t, x, \xi) \equiv\left(\omega, t, x, \xi^{1}, \ldots, \xi^{n}\right) \in \Omega \times Q \times \mathbb{R}^{n}.
$$
\end{itemize}
Let other coefficients and source terms in the equation of system \eqref{p1} satisfy
\begin{equation*}\!\!\!\!\!\!\!\!\!\text{(H3)}
\quad\begin{array}{l}
\ds b_{1} \in L_{\mathbb{F}}^{\infty}\left(0, T ; W^{1,\infty}\left(G ; \mathbb{R}^{n}\right)\right), b_{2} \in L_{\mathbb{F}}^{\infty}\left(0, T; L^{\infty}(G)\right), b_{3} \in L_{\mathbb{F}}^{\infty}\left(0, T ; W^{1, \infty}(G)\right), \\
\ns\ds f \in L_{\mathbb{F}}^{2}\left(0, T ; L^{2}(G)\right) \  \text { and } \  g \in L_{\mathbb{F}}^{2}\left(0, T ; H^{1}_0(G)\right).
\end{array}
\end{equation*}

For readers' convenience, let us first recall the definition of the weak and strong solution to \eqref{p1}.

\begin{definition}\label{weakln} A process $u \in L^2_\dbF(\Om;C([0,T];L^2(G)))\cap L^2_\dbF(0,T;H^1_0(G))$ is said to be a strong solution of
	equation \eqref{p1} if for any $t \in[0, T]$ and $\f\in H_0^1(G)$, it holds that
	\begin{equation*}\label{weaksol}
		\begin{array}{ll} 
			\ds \int_G u(t) \f dx\!=\3n &\ds\int_G\! u_{0} \f dx+\!\int_{0}^{t}\!\int_G\!\Big(\sum_{i, j=1}^{n} a^{i j} u_{i} \f_{j}\!-\! \div b_{1} u \f \!- \! u  b_{1} \cd\nabla\f u \!+\!b_{2} u \f\! +\! f \f \Big) \mathrm{d}x\mathrm{d} s \\
			\ns&\ds +\int_{0}^{t} \int_{G}\left(b_{3}  u +g \right)\f \mathrm{d}x\mathrm{d} W(s), \quad \dbP\mbox{\rm-a.s.}
		\end{array}
	\end{equation*}
\end{definition}

\begin{definition}\label{strongsoln} A process $u \in L_{\mathbb{F}}^{2}\left(\Omega ; C\left([0, T] ; H^{2}(G) \cap H_{0}^{1}(G)\right)\right)$ is said to be a strong solution of
equation \eqref{p1} if for any $t \in[0, T]$ it holds that
\begin{equation}\label{strongsol}
\begin{array}{c}
\ds u(t)=u_{0}+\int_{0}^{t}\[-\sum_{i, j=1}^{n}\left(a^{i j} u_{i} \right)_{j}+ b_{1} \cd\nabla u +b_{2} u +f \] \mathrm{d} s \\
\ns\ds +\int_{0}^{t} \left(b_{3} u +g \right) \mathrm{d}W(s) \mbox{ in } L^2(G), \q \dbP\mbox{\rm-a.s.}
\end{array}
\end{equation}
\end{definition}

Under (H1)--(H3), by the classical well-posedness result for stochastic parabolic equations, we know that \eqref{p1} admits a unique weak solution $u$ (e.g. \cite{Krylov1994}). Furthermore, if $u_{0}\in H^{2}(G) \cap H_{0}^{1}(G)$,
we know that \eqref{p1} admits a unique strong solution $u$ (e.g. \cite{Krylov1994}).

The inverse problem ({\bf IPD}) associated to the
equation \eqref{p1} is as follows.

\vspace{0.1cm}
\begin{itemize}
  	
	\item {\bf Conditional Stability}. Assume that two
	solutions $u$ and $\hat u$ (to the equation
	\eqref{p1}) are given. Let $u(T)$ and
	$ \hat u(T)$ be the
	corresponding terminal values. Can we find a
	positive constant $C$ such that
	\begin{equation}\label{1.18-eq1}
	|u-\hat u| \leq C|\!| u(T)- \hat u|\!|,
	\end{equation}
	with appropriate norms in both sides?
	
	\vspace{0.1cm}
	
	\item {\bf Reconstruction}. Is it possible to
	reconstruct $u$   from the
	terminal value $u(T)$?
	
\end{itemize}
Here and in the rest of this paper, we use $C$
to denote a generic positive constant depending
on $G$, $T$ and $a^{ij}$($i,j=1,\cds,n$) (unless otherwise
stated), which may change from line to line.

\begin{remark}
It is well known that the inverse problem ({\bf IPD}) is ill-posed: small errors in data
may cause huge deviations in solutions. Fortunately, if we assume an a priori bound
for $u(\cd,0)$, then we can restore the stability. This is the reason we consider  conditional stability.
\end{remark} 

Inverse problem  in the above type   is studied extensively for deterministic parabolic equations (see
\cite{KY2019IP,Yamamoto} and the rich references
therein). However, the stochastic case attracts
very little attention. To our best knowledge,
\cite{barbu2,L2012IP} are the only two published
papers addressing this topic.  In \cite{barbu2}, the author study the problem by transform the equation \eqref{p1} to a parabolic equation with random coefficients. Then they can obtain pathwisely a logarithmic convexity property, from which the uniqueness follows. To apply the strategy in \cite{barbu2}, one needs some further assumptions on $b_3$, such as $b_3$ is independent of $x$ (in \cite{barbu2}, the authors assume that $b_3=1$).  In \cite{L2012IP}, the author use a stochastic global Carleman estimate to study the inverse problem ({\bf IPD}). The weight function in that Carleman estimate is a double-exponential function. In such case, the constant $C$ in \eqref{1.18-eq1} will depend on the norm of $b_1$, $b_2$ and $b_3$ double-exponentially. In this paper, we establish a new Carleman estimate with an exponential weight function. Then we improve the conditional stability in \cite{L2012IP} (see Theorem \ref{reguest} for the detail). More importantly, both \cite{barbu2,L2012IP} only address the first question of  the inverse problem ({\bf IPD}). To our best knowledge, this paper is the first one addressing the reconstruction question of the inverse problem ({\bf IPD}).

\ms

As we said before, the main tool for establishing the conditional stability is a new global Carleman estimate for \eqref{p1}.   
Carleman estimates are widely applied to many inverse problems for deterministic partial differential equations (see \cite{BY2017book,FLZ2019book,KT2012book} and the rich references therein).  Recently, Carleman estimates are also introduced to solve inverse problems for stochastic partial differential equations. Particularly, we refer the readers to  \cite{Li2013,L2012IP,WCW2020IP,Y2017IP,Y2021JIIP} for some recent works on inverse problems of stochastic parabolic equations via Carleman estimates. 
In all the above mentioned papers, Carleman estimates are established by using two-layer weight functions.  
In this paper, we improve the Carleman estimates in the previous works for stochastic parabolic equation, i.e., establish a new Carleman estimate for stochastic parabolic equation by applying a one-layer weight function, which is first introduced in \cite{KY2019IP} for backward problem of deterministic heat equation. 
Since the solution of a stochastic parabolic equation is not differentiable with respect to the temporal variable, new difficulties occur in the study of inverse problems for stochastic parabolic equations. 
Thus, the Carleman estimate we obtained is not a trivial extension of \cite{KY2019IP} and other existing works, although the form of the Carleman estimate is similar to the one in \cite{KY2019IP}. 

\ms

Once the Carleman estimate is established, we employ Tikhonov regularization method to reconstruct the solution to \eqref{p1}. Tikhonov regularization method is a very useful tool to solve inverse and ill-posed problems for deterministic partial differential equations (e.g., \cite{LP2013,Schuster2012}). 
To the authors' knowledge, regularization methods for inverse problems of stochastic parabolic equations are rarely studied. The only published work  is \cite{CWC2021IPI}, in which  Cauchy problem for stochastic parabolic equations is studied by a Kaczmarz method.   In this paper, we propose a regularized solution by minimising a functional, which is giving based on the Tikhonov regularization.  The analysis of the existence and uniqueness of the regularized solution, and proof for error estimate between the regularized solution and the exact solution under an a-proior condition are also present. 

In this paper, we also give a numerical solution to the inverse problem  ({\bf IPD}), i.e., we numerically solve the reconstruction problem   by the proposed regularization method, with the combination of the conjugate gradient method to the Tikhonov type functional and Picard iteration.  We mention that, in numerical solving the forward-backward stochastic parabolic equation, we apply the idea given in \cite{DP2016SISC}, where the discretization is given based on the implicit Euler method for a temporal discretization and a least squares Monte Carlo method in combination with a stochastic gradient method. Compared with deterministic setting, this is also a topic far from well studied.  As far as we know, there is no published work  addressing that.  

The rest of the paper is organized as follow. 
We first establish a global Carleman estimate for the stochastic parabolic equation  \eqref{p1} in Section 2,  and prove the conditional stability for the inverse problem under an a-priori information in Section 3. 
A regularization solution based on the Tikhonov regularization method is proposed in Section 4. 
At last, we give numerical algorithm to the problem, and illustrate the approximations to several  examples in both one and two dimensional spatial domains in Section 5.

\section{Carleman estimate}

This section is devoted to establishing a Carleman estimate for \eqref{p1}. We first give the weight function. For $\l>0$, set
\begin{equation}
\psi=(t+1)^\l,\qq \f=e^\psi.
\end{equation}
The Carleman estimate is as follows:
\begin{theorem}\label{carleman}
For any $\e \in[0, T),$ there exists a $\l_{0}>0$ such that for all $\l \geqslant \l_{0},$ there holds that
\begin{align}\label{519eq3}
&\mathbb{E}\int_\e^T\int_G \l^2 (t+1)^{(\l-2)}v^2\mathrm{d} t\mathrm{d}x+\mathbb{E}\int_\e^T\int_G \l(t+1)^{-1}|\nabla v|^2\mathrm{d} t\mathrm{d}x \nonumber\\
\ns\leq&C\[\l(T\!+\!1)^{\l-1}\mathbb{E}\|v(T)\|^2_{L^2(G)}\!+\! \mathbb{E}\|v(\e)\|^2_{L^2(G)} \!+\! \mathbb{E}\! \int_{\e}^{T} \!\int_{G}\! \f^2(|\nabla g|^2\! +\!g^{2}\!+\!f^2) \mathrm{d} x \mathrm{d} t\].
\end{align}
\end{theorem}

\begin{proof}
We divide the proof into three steps.

\noindent\textit{Step 1. In this step, we establish a weighted  identity.}

Let $v\triangleq \varphi u$. Then we know that
\begin{align}
 \mathrm{d} u-\sum_{i, j=1}^{n}(a^{i j}u_i)_j \mathrm{d} t = \f^{-1}\[\mathrm{d}v-\l(t+1)^{(\l-1)}v\mathrm{d} t-\sum_{i, j=1}^{n}(a^{i j} v_{i})_j \mathrm{d} t\].
\end{align}
Consequently,
\begin{align}\label{512eq1}
\ds&-\f\[\l(t+1)^{(\l-1)}v+\sum_{i, j=1}^{n}(a^{i j} v_{i})_j\]\[\mathrm{d} u-\sum_{i, j=1}^{n}(a^{i j} u_i)_j \mathrm{d} t\]\nonumber\\
\ns\ds= &-\l(t+1)^{(\l-1)}v\mathrm{d}v-\sum_{i, j=1}^{n}(a^{i j} v_{i})_j\mathrm{d}v+\[\l(t+1)^{(\l-1)}v+\sum_{i, j=1}^{n}(a^{i j} v_{i})_j\]^2\mathrm{d} t.
\end{align}

According to It\^o's formula, we know
\begin{align}\label{512eq3}
\ds&-\l(t+1)^{(\l-1)}v\mathrm{d} v=-
\frac12\l(t+1)^{(\l-1)}\mathrm{d} v^2+\frac12\l(t+1)^{(\l-1)}(\mathrm{d} v)^2\nonumber\\
\ns\ds=&-\frac12\mathrm{d}\big[\l(t+1)^{(\l-1)} v^2\big]+\frac12\l(\l-1)(t+1)^{(\l-2)} v^2\mathrm{d}t+\frac12\l(t+1)^{(\l-1)}(\mathrm{d} v)^2.
\end{align}
Noting  $a^{ij}=a^{ji}$ for $i,j=1,\cds,n$, we get that
\begin{align}\label{512eq2}
\ds-\sum_{i, j=1}^{n}(a^{i j} v_{i})_j\mathrm{d}v=&-\sum_{i, j=1}^{n}(a^{i j} v_{i}\mathrm{d}v)_j+\frac12 \mathrm{d}\(\sum_{i, j=1}^{n}a^{i j} v_{i}v_j\)\nonumber\\
\ns\ds&-\frac12 \sum_{i, j=1}^{n}a^{i j}_t v_{i}v_j\mathrm{d}t-\frac12 \sum_{i, j=1}^{n}a^{i j} \mathrm{d}v_{i}\mathrm{d}v_j.
\end{align}
%

Substituting \eqref{512eq2} and \eqref{512eq3} into \eqref{512eq1}, we obtain that
\begin{align}\label{512eq4}
\ds&-\f\[\l(t+1)^{(\l-1)}v+\sum_{i, j=1}^{n}(a^{i j} v_{i})_j\]\[\mathrm{d} u-\sum_{i, j=1}^{n}(a^{i j} u_{i})_j \mathrm{d} t\]\nonumber\\
\ns\ds=&-\sum_{i, j=1}^{n}(a^{i j} v_{i}\mathrm{d}v)_j+\frac12 \mathrm{d}\[\sum_{i, j=1}^{n}a^{i j} v_{i}v_j-\l(t+1)^{(\l-1)} v^2\]-\frac12 \sum_{i, j=1}^{n}a^{i j} \mathrm{d}v_{i}\mathrm{d}v_j\\
\ns\ds&+\frac12 \[-\sum_{i, j=1}^{n}a^{i j}_t v_{i}v_j+\l(\l-1)(t+1)^{(\l-2)} v^2\]\mathrm{d}t+\frac12\l(t+1)^{(\l-1)}(\mathrm{d} v)^2\nonumber\\
\ns\ds&+\[\l(t+1)^{(\l-1)}v+\sum_{i, j=1}^{n}a^{i j} v_{ij}\]^2\mathrm{d} t.\nonumber
\end{align}
Moreover,
\begin{align}\label{514eq2}
&  \frac14\l (t+1)^{-1}\f v\[\mathrm{d} u-\sum_{i, j=1}^{n}(a^{i j} u_i)_j \mathrm{d} t\]\nonumber\\
\ns=&\frac18\l \mathrm{d}\big[(t+1)^{-1}v^2\big]+\frac18\l(t+1)^{-2}v^2dt-\frac18\l(t+1)^{-1} (\mathrm{d}v)^2-\frac14\l^2(t+1)^{(\l-2)}v^2\mathrm{d} t\nonumber\\
\ns&- \frac14(t+1)^{-1}\l\sum_{i, j=1}^{n}(a^{i j} v_{i}v)_j \mathrm{d} t+\frac14(t+1)^{-1}\l\sum_{i, j=1}^{n}a^{i j} v_{i}v_j \mathrm{d} t.
\end{align}

Sum \eqref{512eq4} and \eqref{514eq2} up, we obtain that
\begin{align}\label{514eq3}
&-\f\[\l(t+1)^{(\l-1)}v+\sum_{i, j=1}^{n}a^{i j} v_{ij}+\frac14\l (t+1)^{-1}v\] \[\mathrm{d} u-\sum_{i, j=1}^{n}(a^{i j} u_i)_j \mathrm{d} t\]\nonumber\\
\ns=&-\sum_{i, j=1}^{n}\[a^{i j} v_{i}\mathrm{d}v+\frac14(t+1)^{-1}\l a^{i j} v_{i}v\]_j-\frac12 \sum_{i, j=1}^{n}a^{i j} \mathrm{d}v_{i}\mathrm{d}v_j\nonumber\\
\ns&+\frac12 \mathrm{d}\[\sum_{i, j=1}^{n}a^{i j} v_{i}v_j-\l(t+1)^{(\l-1)} v^2+\frac14\l(t+1)^{-1} v^2\]\\
\ns&+\frac12 \[-\sum_{i, j=1}^{n}a^{i j}_t v_{i}v_j+\l(\l-1)(t+1)^{(\l-2)} v^2-\frac12\l^2(t+1)^{(\l-2)}v^2+\frac14\l(t+1)^{-2}v^2\]\mathrm{d}t\nonumber\\
\ns&+\frac12\[\l(t+1)^{(\l-1)}-\frac14\l(t+1)^{-1}\](\mathrm{d} v)^2+\[\l(t+1)^{(\l-1)}v+\sum_{i, j=1}^{n}a^{i j} v_{ij}\]^2\mathrm{d} t\nonumber\\
\ns\ds&
+\frac14\l(t+1)^{-1}\sum_{i, j=1}^{n}a^{i j} v_{i}v_j \mathrm{d} t.\nonumber
\end{align}

Applying \eqref{514eq3} to the equation \eqref{p1}, integrating \eqref{514eq3} on $[\e,T]\times\Omega$ for $\e\in[0,T)$ and taking mathematical expectation, we have
\begin{align}\label{514eq4}
&-\mathbb{E}\int_\e^T\int_G \f\[\l(t+1)^{(\l-1)}v+\sum_{i, j=1}^{n}a^{i j} v_{ij}+\frac14\l (t+1)^{-1}v\] \[\mathrm{d} u-\sum_{i, j=1}^{n}(a^{i j} u_i)_j \mathrm{d} t\]\mathrm{d}x\nonumber\\
=&-\mathbb{E}\int_\e^T\int_G \sum_{i, j=1}^{n}\[a^{i j} v_{i}\mathrm{d}v+\frac14(t+1)^{-1}\l a^{i j} v_{i}v\mathrm{d} t\]_j\mathrm{d}x\nonumber\\
&+\frac12 \mathbb{E}\int_\e^T\int_G \mathrm{d}\[\sum_{i, j=1}^{n}a^{i j} v_{i}v_j-\l(t+1)^{(\l-1)} v^2+\frac14\l(t+1)^{-1} v^2\]\mathrm{d}x\nonumber\\
&+\frac12 \mathbb{E}\int_\e^T\int_G \[\l(\l-1)(t+1)^{(\l-2)} v^2-\frac12\l^2(t+1)^{(\l-2)}v^2+\frac14\l(t+1)^{-2}v^2\]\mathrm{d} t\mathrm{d}x\nonumber\\
&+\frac12\mathbb{E}\int_\e^T\int_G \[\l(t+1)^{(\l-1)}-\frac14\l(t+1)^{-1}\](\mathrm{d} v)^2\mathrm{d}x\\
&+\mathbb{E}\int_\e^T\int_G \[\l(t+1)^{(\l-1)}v+\sum_{i, j=1}^{n}a^{i j} v_{ij}\]^2\mathrm{d} t\mathrm{d}x\nonumber\\
&-\frac12 \mathbb{E}\int_\e^T\int_G \[\sum_{i, j=1}^{n}a^{i j} \mathrm{d}v_{i}\mathrm{d}v_j+\sum_{i, j=1}^{n}a^{i j}_t v_{i}v_j\mathrm{d} t-\frac12(t+1)^{-1}\l\sum_{i, j=1}^{n}a^{i j} v_{i}v_j \mathrm{d} t\]\mathrm{d}x\nonumber\\
\ns=:&\sum_{j=1}^6 I_j.\nonumber
\end{align}

\noindent\textit{Step 2. In this step, we estimate right hand side of \eqref{514eq4}.}

Since $u\big|_{\Sigma}=0$, $v\big|_\Sigma=0$, by the divergence theorem,  we have
\begin{align}\label{513eq2}
I_1&=-\mathbb{E}\int_\e^T\int_G\sum_{i, j=1}^{n}\[a^{i j} v_{i}\mathrm{d}v+\frac14(t+1)^{-1}\l a^{i j} v_{i}v \mathrm{d} t\]_j\mathrm{d}x =0.
\end{align}

By the assumptions (H1) and (H2), we get
\begin{align}\label{513eq3}
\ds I_2=&\frac12 \mathbb{E}\int_\e^T\int_G\mathrm{d}\[\sum_{i, j=1}^{n}a^{i j} v_{i}v_j-\l(t+1)^{(\l-1)} v^2+\frac14\l(t+1)^{-1} v^2\]\mathrm{d}x\nonumber\\
=&\frac12 \mathbb{E}\int_G\[\sum_{i, j=1}^{n}a^{i j} v_{i}v_j-\l(t+1)^{(\l-1)} v^2+\frac14\l(t+1)^{-1} v^2\]_{t=\delta}^T\mathrm{d}x\nonumber\\
=&\frac12 \mathbb{E}\int_G\[\sum_{i, j=1}^{n}a^{i j} v_{i}(T)v_j(T)-\l(t+1)^{(\l-1)} v^2(T)+\frac14\l(t+1)^{-1} v^2(T)\]\mathrm{d}x\\
&-\frac12 \mathbb{E}\int_G\[\sum_{i, j=1}^{n}a^{i j} v_{i}(\e)v_j(\e)-\l(t+1)^{(\l-1)} v^2(\e)+\frac14\l(t+1)^{-1} v^2(\e)\]\mathrm{d}x\nonumber\\
\geq&\frac12 \mathbb{E}\int_G\sigma |\nabla v(T)|^2\mathrm{d}x-\l(t+1)^{(\l-1)} v(T)^2\mathrm{d}x-\frac12 \mathbb{E}\int_G\sum_{i, j=1}^{n}a^{i j} v_{i}(\e)v_j(\e)\mathrm{d}x\nonumber\\
&+\frac12 \mathbb{E}\int_G\l(t+1)^{(\l-1)}\[1-\frac14(\e+1)^{-\l}\] v^2(\e)\mathrm{d}x.\nonumber
\end{align}
Let us choose $\l\geq1$. Then $1-\frac14(\e+1)^{-\l}\geq\frac34>0$, and thus
\begin{align}\label{615eq1}
\ds I_2
\geq&-\frac12 \mathbb{E}\int_G\l(T+1)^{\l-1} v(T)^2\mathrm{d}x-\frac12 \mathbb{E}\int_G\sum_{i, j=1}^{n}a^{i j} v_{i}(\e)v_j(\e)\mathrm{d}x\nonumber\\
\geq &-\frac12\mathbb{E}\l(T+1)^{\l-1}\|v(T)\|^2_{L^2(G)}-\frac12\sum_{i, j=1}^{n}\|a^{ij}(\e)\|^2_{L^\infty_{\cF_\e}(\Omega;L^\infty(G))}\mathbb{E}\|\nabla v(\e)\|^2_{L^2(G)}.
\end{align}

Since
\begin{align}\label{513eq4}
\ds&-\frac12\mathbb{E} \int_{\e}^{T} \int_G\sum_{i, j=1}^{n} a^{i j} \mathrm{d} v_{i} \mathrm{d} v_{j} \mathrm{d} x\nonumber\\
=&-\frac{1}{2} \mathbb{E} \int_{\e}^{T} \int_G \f^2\sum_{i, j=1}^{n} a^{i j} \left(b_{3} u+g\right)_{i}\left(b_{3} u+g\right)_{j} \mathrm{d} x \mathrm{d} t\nonumber\\
\geq&-C\mathbb{E} \int_{\e}^{T} \int_{G}\left(b_{3}^{2}|\nabla v|^{2}+\left|\nabla b_{3}\right|^{2} v^{2}+\f^2|\nabla g|^{2}+\f^2|g|^{2}\right) \mathrm{d} x \mathrm{d} t\\
\geq&-C\left(\|b_{3}\|_{L_{\mathbb{F}}^{\infty}\left(0, T ; W^{1, \infty}(\Omega)\right)}^{2}+1\right) \mathbb{E} \int_{\e}^{T} \int_{\Omega}\left(|\nabla v|^{2}+v^{2}\right) \mathrm{d} x \mathrm{d} t\nonumber\\
&\quad -C \mathbb{E} \int_{\e}^{T} \int_G\f^2\left(|\nabla g|^{2}+g^{2}\right) \mathrm{d} x \mathrm{d} t,\nonumber
\end{align}
we obtain
\begin{eqnarray}\label{515eq1}
 I_6\3n&=\3n&-\frac12 \mathbb{E}\int_\e^T\int_G \[\sum_{i, j=1}^{n}a^{i j} \mathrm{d}v_{i}\mathrm{d}v_j+\sum_{i, j=1}^{n}a^{i j}_t v_{i}v_j\mathrm{d} t-\frac12\l(t+1)^{-1}\sum_{i, j=1}^{n}a^{i j} v_{i}v_j \mathrm{d} t\]\mathrm{d}x\nonumber\\
&\3n\geq\3n&-C \mathbb{E} \int_{\e}^{T} \int_{G}\left(\|b_{3}\|_{L_{\mathbb{F}}^{\infty}\left(0, T ; W^{1, \infty}(\Omega)\right)}^{2}+1\right)\left(|\nabla v|^{2}+v^{2}\right) \mathrm{d} x \mathrm{d} t \\
&&-C \mathbb{E}\! \int_{\e}^{T} \!\int_G \f^2\left(|\nabla g|^{2}+g^{2}\right) \mathrm{d} x \mathrm{d} t +\mathbb{E}\!\int_\e^T\!\int_G \[-C|\nabla v|^2\!+\frac14\l(t\!+\!1)^{-1}\sigma|\nabla v|^2\]\mathrm{d}x\mathrm{d} t.
\nonumber
\end{eqnarray}

By direct computations, we see that
\begin{align}\label{514eq1}
I_4=&\frac12\mathbb{E}\int_\e^T\int_G \[\l(t+1)^{(\l-1)}-\frac14\l(t+1)^{-1}\](\mathrm{d} v)^2\mathrm{d}x\nonumber\\
=& \frac12\mathbb{E}\int_\e^T\int_G\[\l(t+1)^{(\l-1)}-\frac14\l(t+1)^{-1}\]e^{2(t+1)^\l}(b_3u+g)^2\mathrm{d}x\\
\geq&-\mathbb{E}\int_\e^T\int_G\frac14\l(t+1)^{-1}\[(b_3v)^2+e^{2(t+1)^\l}g^2\]\mathrm{d} t\mathrm{d}x.\nonumber
\end{align}
For  $\l\geq\max\{3,T+1\}$, we have
\begin{align}\label{515eq2}
I_3+I_4=&\frac12 \mathbb{E}\int_\e^T\int_G \[\l(\l-1)(t+1)^{(\l-2)} v^2-\frac12\l^2(t+1)^{(\l-2)}v^2+\frac14\l(t+1)^{-2}v^2\]\mathrm{d} t\mathrm{d} x\nonumber\\
&+\frac12\mathbb{E}\int_\e^T\int_G\[\l(t+1)^{(\l-1)}-\frac14\l(t+1)^{-1}\](\mathrm{d} v)^2\mathrm{d}x\nonumber\\
\geq& \frac13\mathbb{E}\int_\e^T\int_G\l^2(t+1)^{(\l-2)}v^2\mathrm{d}x\mathrm{d} t+O(\l)\mathbb{E}\int_\e^T\int_G v^2\mathrm{d} t\mathrm{d}x\\
\ns&-\frac14\mathbb{E}\int_\e^T\int_G\l(t+1)^{-1}\f^2g^2\mathrm{d}x\mathrm{d} t.\nonumber
\end{align}

Substituting \eqref{513eq2}--\eqref{515eq2} into \eqref{514eq4}, we get
\begin{eqnarray}\label{515eq3}
&&-\mathbb{E}\!\int_\e^T\!\int_G \Big\{\f\[\l(t+1)^{(\l-1)}v+\sum_{i, j=1}^{n}a^{i j} v_{ij}\!+\frac14\l(1+t)^{-1} v\] \[\mathrm{d} u-\!\sum_{i, j=1}^{n}\!(a^{i j} u_i)_j \mathrm{d} t\]\Big\}\mathrm{d}x\nonumber\\
&\geq\3n& -\frac12\mathbb{E}\[\l(T+1)^{\l-1}\|v(T)\|^2_{L^2(G)}+\sum_{i,j=1}^n\|a^{ij}\|^2_{L^\infty_{\cF_\e}(\Omega;W^{1,\infty}(G))}\|\nabla v(\e)\|^2_{L^2(G)}\]\nonumber\\
&&-\mathbb{E} \int_{\e}^{T} \int_{\Omega}\big(\|b_{3}\|_{L_{\mathbb{F}}^{\infty}\left(0, T ; W^{1, \infty}(\Omega)\right)}^{2}+1\big)\left(|\nabla v|^{2}+v^{2}\right) \mathrm{d} x \mathrm{d} t\\
&&-C \mathbb{E} \!\int_{\e}^{T}\! \int_G \f^2\left(|\nabla g|^{2}\!+g^{2}\right) \mathrm{d} x \mathrm{d} t-\mathbb{E}\!\int_\e^T\!\int_G \[C|\nabla v|^2\!-\frac14\l(t\!+1)^{-1}\sigma|\nabla v|^2\]\mathrm{d}x\mathrm{d} t\nonumber\\
&&+\frac13\mathbb{E}\int_\e^T\int_G\l^2(t+1)^{(\l-2)}v^2\mathrm{d}x\mathrm{d} t+O(\l)\mathbb{E}\int_\e^T\int_G v^2\mathrm{d} t\mathrm{d}x\nonumber\\
&&-\frac14\mathbb{E}\int_\e^T\int_G\l(t+1)^{-1}\f^2g^2\mathrm{d}x\mathrm{d} t+\mathbb{E}\int_\e^T\int_G \[\l(t+1)^{(\l-1)}v+\sum_{i, j=1}^{n}a^{i j} v_iv_j\]^2\mathrm{d} t\mathrm{d}x.\nonumber
\end{eqnarray}

\noindent\textit{Step 3. In this step, we estimate left hand side of \eqref{514eq4}, and deduce Carleman inequality \eqref{519eq3}.}

From \eqref{p1}, we know
\begin{align}\label{515eq5}
&-\mathbb{E}\int_\e^T\int_G \f\[\l(t+1)^{(\l-1)}v+\sum_{i, j=1}^{n}a^{i j} v_{ij}+\frac14\l(t+1)^{-1} v\] \[\mathrm{d} u-\sum_{i, j=1}^{n}(a^{i j} u_i)_j \mathrm{d} t\]\mathrm{d}x\nonumber\\
=&-\mathbb{E}\int_\e^T\int_G \f\[\l(t+1)^{(\l-1)}v+\sum_{i, j=1}^{n}a^{i j} v_{ij}+\frac14\l(t+1)^{-1} v\] \left( b_{1}\cd\nabla u +b_{2} u+f\right) \mathrm{d} t\mathrm{d}x\nonumber\\
\leq&\mathbb{E}\int_\e^T\int_G \[\l(t+1)^{(\l-1)}v+\sum_{i, j=1}^{n}a^{i j} v_{ij}\]^2 \mathrm{d} t\mathrm{d}x+\frac1{16}\mathbb{E}\int_\e^T\int_G \[\frac14\l(t+1)^{-1} v\]^2\mathrm{d} t\mathrm{d}x\nonumber\\
&+\frac{17}4\mathbb{E}\int_\e^T\int_G e^{2(t+1)^\l} \left(b_{1}\cd\nabla u+b_{2} u+f\right)^2 \mathrm{d} t\mathrm{d}x\\
\leq& \mathbb{E}\int_\e^T\int_G \[\l(t+1)^{(\l-1)}v+\sum_{i, j=1}^{n}a^{i j} v_{ij}\]^2 \mathrm{d}x\mathrm{d} t+\frac1{256}\mathbb{E}\int_\e^T\int_G \l^2(t+1)^{-2} v^2\mathrm{d} t\mathrm{d}x\nonumber\\
&+\frac{51}4\|b_{1}\|^2_{L^\infty_{\mathbb{F}}(0,T;L^\infty(G,\mathbb{R}^n))}\mathbb{E}\int_\e^T\int_G |\nabla v|^2\mathrm{d}x\mathrm{d} t\nonumber\\
\ns&+\frac{51}4\|b_{2}\|^2_{L^\infty_{\mathbb{F}}(0,T;L^\infty(G))}\mathbb{E}\int_\e^T\int_G |v|^2\mathrm{d}x\mathrm{d} t+\frac{51}4\mathbb{E}\int_\e^T\int_G \f^2 |f|^2\mathrm{d}x\mathrm{d} t.\nonumber
\end{align}

Combining \eqref{515eq3} and \eqref{515eq5}, we know that there exists $\l_1>0$ such that for all $\l\geq \l_1$, there holds
\begin{align}\label{515eq6}
&-\frac12\mathbb{E}\l(T+1)^{\l-1}\|v(T)\|^2_{L^2(G)}-\frac12\sum_{i,j=1}^n\|a^{ij}(\e)\|^2_{L^\infty_{\cF_\e}(\Omega;L^\infty(G))}\mathbb{E}\|\nabla v(\e)\|^2_{L^2(G)}\nonumber\\
&-C\mathbb{E} \int_{\e}^{T} \int_{G}\left(\|b_{3}\|_{L_{\mathbb{F}}^{\infty}\left(0, T ; W^{1, \infty}(\Omega)\right)}^{2}+1\right)\left(|\nabla v|^{2}+v^{2}\right) \mathrm{d} x \mathrm{d} t\nonumber\\
&-C \mathbb{E} \int_{\e}^{T} \int_G \f^2\left(|\nabla g|^{2}+g^{2}\right) \mathrm{d} x \mathrm{d} t -\mathbb{E}\int_\e^T\int_G \[C|\nabla v|^2-\frac14\l(t+1)^{-1}\sigma|\nabla v|^2\]\mathrm{d}x\mathrm{d} t\nonumber\\
&+\frac13\mathbb{E}\int_\e^T\int_G\l^2(t+1)^{(\l-2)}v^2\mathrm{d}x\mathrm{d} t+O(\l)\mathbb{E}\int_\e^T\int_G v^2\mathrm{d} t\mathrm{d}x \\
&-C\mathbb{E}\int_\e^T\int_G\l(t+1)^{(\l-1)}\f^2g^2\mathrm{d}x\mathrm{d} t +\mathbb{E}\int_\e^T\int_G \[\l(t+1)^{(\l-1)}v+\sum_{i, j=1}^{n}a^{i j} v_iv_j\]^2\mathrm{d} t\mathrm{d}x\nonumber\\
\leq&\mathbb{E}\int_\e^T\int_G \[\l(t+1)^{(\l-1)}v+\sum_{i, j=1}^{n}a^{i j} v_{ij}\]^2 \mathrm{d}x\mathrm{d} t+\frac1{256}\mathbb{E}\int_\e^T\int_G \l^2(t+1)^{-2} v^2\mathrm{d} t\mathrm{d}x\nonumber\\
&+\frac{51}4\|b_{1}\|^2_{L^\infty_{\mathbb{F}}(0,T;L^\infty(G,\mathbb{R}^n))}\mathbb{E}\int_\e^T\int_G |\nabla v|^2\mathrm{d}x\mathrm{d} t\nonumber\\
&+\frac{51}4\|b_{2}\|^2_{L^\infty_{\mathbb{F}}(0,T;L^\infty(G,\mathbb{R}^n))}\mathbb{E}\int_\e^T\int_G |v|^2\mathrm{d}x\mathrm{d} t+\frac{51}4\mathbb{E}\int_\e^T\int_G \f^2|f|^2\mathrm{d}x\mathrm{d} t.\nonumber
\end{align}
Let $\l_2= \max\{\l_1,C(\|b_{3}\|_{L_{\mathbb{F}}^{\infty}\left(0, T ; W^{1, \infty}(\Omega)\right)}^{2}+1)\}$. 
Then, for all $\l\geq \l_2$,
\begin{align}\label{515eq6-1}
&\frac13\mathbb{E}\int_\e^T\int_G\l^2\[(t+1)^{(\l-2)}-\frac1{256}\]v^2\mathrm{d}x\mathrm{d} t+O(\l)\mathbb{E}\int_\e^T\int_G v^2\mathrm{d} t\mathrm{d}x\nonumber\\
&+\!\mathbb{E}\!\int_\e^T\!\!\!\int_G\! \[\frac14\l\sigma(t\!+\!1)^{-1}\!\! -\!C(\|b_{3}\|_{L_{\mathbb{F}}^{\infty}(0, T ; W^{1, \infty}(G))}^{2}\!\!+\!1) \!-\!\frac{51}4\|b_{1}\|^2_{L^\infty_\dbF(0,T;L^\infty(G,\mathbb{R}^n))} \]|\nabla v|^2\mathrm{d} t\mathrm{d}x\nonumber\\
\leq&\frac12\mathbb{E}\l(T+1)^{\l-1}\|v(T)\|^2_{L^2(G)}+\frac12\sum_{i,j=1}^n\|a^{ij}(\e)\|^2_{L^\infty_{\cF_\e}(\Omega;L^\infty(G))}\mathbb{E}\|v(\e)\|^2_{L^2(G)}\\
&+C\mathbb{E}\! \int_{\e}^{T}\!\! \int_G \!\f^2\left[ |\nabla g|^{2}\!+\(1 \!+\!\l(t\!+\!1)^{(\l-1)}\)g^{2}\right]\mathrm{d} x \mathrm{d} t \!+\!\frac{51}4\mathbb{E}\!\int_\e^T\!\!\int_G\! e^{2(t+1)^\l}|f|^2\mathrm{d}x\mathrm{d} t.\nonumber
\end{align}
Moreover, let 
$$\l_0= \max\Big\{\l_2,\frac43\sigma^{-1}(T+1)\[C+C(\|b_{3}\|_{L_{\mathbb{F}}^{\infty}(0, T ; W^{1, \infty}(G))}^{2}+1)+\frac{51}4\|b_{1}\|^2_{L^\infty_\dbF(0,T;L^\infty(G,\mathbb{R}^n))}\]\Big\}.$$ 
For all $\l\geq \l_0$, we have 
\begin{align}
&\mathbb{E}\int_\e^T\int_G \l^2 (t+1)^{(\l-2)}v^2\mathrm{d} t\mathrm{d}x+\mathbb{E}\int_\e^T\int_G \l(t+1)^{-1}|\nabla v|^2\mathrm{d} t\mathrm{d}x\nonumber\\
\leq&\frac12\mathbb{E}\l(T+1)^{\l-1}\|v(T)\|^2_{L^2(G)}+\frac12\sum_{i,j=1}^n\|a^{ij}(\e)\|^2_{L^\infty_{\cF_\e}(\Omega;L^\infty(G))}\mathbb{E}\|\nabla v(\e)\|^2_{L^2(G)}\nonumber\\
&+C\mathbb{E} \int_{\e}^{T} \int_{G} \f^2 (f^2+|\nabla g|^2+g^{2}) \mathrm{d} x \mathrm{d} t.\nonumber
\end{align}
This gives \eqref{519eq3}.
\end{proof}

\section{Conditional Stability}

In this section, we establish a conditional stability of the inverse problem ({\bf IPD}).

Let us first introduce the a priori bound for the initial data. Let  $M>0$ and set
\begin{equation}\label{615eq4-2}
	U_M\= \{\xi\in L^2_{\cF_0}(\Om;H^1(G))| |\nabla \xi|\leq M, \; \dbP\mbox{-a.s.}\}.
\end{equation}

\begin{theorem}\label{reguest}
Let $\delta_0>0$ be sufficiently small such that
	\begin{equation}\label{615eq5}
		\ln\(\ln\(\delta_0^{-\frac13}\)^{\frac{1}{\ln(T+1)}}\)\geq\l_0.
	\end{equation}
Suppose that $u_1,u_2 $ are weak solutions of problem \eqref{p1} with initial data belong to $U_M$ and $||u_1(T)-u_2(T)||_{L^2(G)}=\d\leq \d_0$.
Then the following estimate hold
	\begin{equation}\label{616eq6}
		\|u_1-u_2\|_{L^2_\mathbb{F}(\e,T;H^1(G))}\leq C(M+1)e^{-\frac{1}{3^c}(\ln(||u_1(T)-u_2(T)||_{L^2(G)}^{-1})^c)},
	\end{equation}
	where $C$ is independent of $M$ and $c=c(\e,T)=\frac{\ln(\e +1)}{\ln(T+1)}\in (0,1)$.	 
\end{theorem}

\begin{proof}  Let  $w(x,t)=u_1(x,t)-u_2(x,t)$ and $R(x)=u_1(x,T)-u_2(x,T)$. From \eqref{p1}, we have
 \begin{equation}\label{615eq6}
\left\{\begin{array}{ll}
\ds\mathrm{d} w-\sum_{i, j=1}^{n}(a^{i j} w_{i})_j \mathrm{d} t=\left( b_{1}\cd\nabla w +b_{2} w\right)\mathrm{d} t+b_{3} w \mathrm{d} W(t) & \text { in } Q, \\ 
\ns w(x,t) =0  & \text { on } \Sigma,  \\
\ns w(x,T)=R(x)  & \text { in } G.
\end{array}\right.
\end{equation}

Applying \eqref{519eq3} to \eqref{615eq6}, we obtain 
\begin{align}\label{616eq1}
&\frac12\l(T+1)^{\l-1}e^{2(T+1)^\l}\mathbb{E}\|w(T)\|^2_{L^2(G)}+\frac12\sum_{i,j=1}^n\|a^{ij}\|^2_{L^\infty_{\cF_\e}(\Omega;L^\infty(G))}e^{2(\e+1)^\l}\mathbb{E}\|w(\e)\|^2_{L^2(G)}\nonumber\\
\geq&\mathbb{E}\int_\e ^T\int_G \l^2 (t+1)^{(\l-2)}\f^2 w^2 \mathrm{d} t\mathrm{d}x+\mathbb{E}\int_\e^T\int_G \l(t+1)^{-1}\f^2|\nabla w|^2 \mathrm{d} t\mathrm{d}x\nonumber\\
\geq&\mathbb{E}\int_\e^T\int_G \l^2 \f^2w^2 \mathrm{d} t\mathrm{d}x+(\e+1)^{-1}\mathbb{E}\int_\e^T\int_G \l\f^2|\nabla w|^2  \mathrm{d} t\mathrm{d}x\\
\geq&\bar{C}e^{2(\e +1)^\l}\mathbb{E}\int_\e^T\int_G \big(w^2+|\nabla w|^2\big) \mathrm{d} t\mathrm{d}x\nonumber\\
=&\bar{C}e^{2(\e +1)^\l}\|w\|_{L^2_\mathbb{F}(\e,T;H^1(G))}.\nonumber
\end{align}
By choosing $\l$ large enough such that $\l(T+1)^{\l-1}e^{2(T+1)^\l}\leq e^{3(T+1)^\l}$, then there hold
\begin{align}\label{616eq2}
\|w\|^2_{L^2_\mathbb{F}(\e,T;H^1(G))}\leq \bar{C}_1e^{3(T+1)^\l}\mathbb{E}\|w(T)\|^2_{L^2(G)}+\bar{C}_2\mathbb{E}\|\nabla w(\e)\|^2_{L^2(G)}.
\end{align}

Choosing 
$\l=\l(\delta)$ 
as
\begin{equation}\label{616eq4}
\l=\l(\delta)=\ln\[\(\ln(\delta^{-\frac13})\)^\frac{1}{\ln(T+1)}\],
\end{equation}
and substituting it in to \eqref{616eq2} we get
\begin{equation}\label{616eq5}
\|w\|^2_{L^2_\mathbb{F}(\e,T;H^1(G)}\leq \bar{C}_1\delta+\bar{C}_2M^2e^{-\frac{2}{3^c}\ln(\delta^{-1})^c}.
\end{equation}
This implies \eqref{616eq6}.
\end{proof}

\begin{remark}\label{rmk3.1}
In Theorem \ref{reguest}, the norm in the left hand side is $\|\cd\|_{L^2_\dbF(\e,T;H^1(G)}$. We can also obtain the estimate at each $t_0\in(0,T)$. The argument is as follows. Choose $\e<t_0$ in Theorem \ref{reguest}. Since 
\begin{equation}\label{ree1}
C\|w(s)\|_{H^1(G)}^2\geq\|w(t_0)\|_{H^1(G)}^2, \qquad \forall \e\leqslant s\leqslant t_0,
\end{equation}
integrating both side of \eqref{ree1} with respect to $s$ in $[\e,t_0]$, we get
\begin{equation}\label{ree2}
C\int_\e^{t_0}\|w(s)\|_{H^1(G)}^2ds\geq (t_0-\e)\|w(t_0)\|_{H^1(G)}^2.
\end{equation}
Combining \eqref{ree2} with \eqref{616eq5}, we can obtain 
\begin{equation*} 
\|w(t_0)\|_{H^1(G)}^2\leq C(t_0-\e)^{-1}\int_0^{t_0}\|w(s)\|_{H^1(G)}^2ds\leq C(t_0-\e)^{-1}(\bar{C}_2M+\bar{C}_1)e^{-\frac{2}{3^c}(\ln(\delta^{-1})^c)}.
\end{equation*}
\end{remark}

\section{Regularization for the reconstruction problem}

In the following, we prove a convergence result by the Carleman estimate \eqref{carleman}. Suppose $u_T \in L^2_{\cF_T}(\Omega;L^2(G))$ is the exact terminal value of the initial-boundary problem \eqref{p1} with the initial datum $u_0\in L^2_{\cF_0}(\Omega;H^2(G)\cap H_0^1(G))$,  and $u_T^\delta\in L^2_{\cF_T}(\Omega;L^2(G)) $ the noisy data satisfying
\begin{equation}\label{620eq2}
\|u_T-u_T^\delta\|_{L^2_{\cF_T}(\Omega;L^2(G))}\leq \delta,
\end{equation} 
where $\delta>0$ is the noise level of data.

Consider  the following Tikhonov type functional on $L^2_{\cF_0}(\Om;H^2(G)\cap H_0^1(G))$:
\begin{align}\label{616eq7}
J_\alpha(y_{0})\! =& \mathbb{E}\! \int_G\! \Big\{
y_{0}+\!\int_{t_0}^{T}\!\[\sum_{i, j=1}^{n}(a^{i j} y_{i})_j\!+\!b_{1}\cd \nabla y \!+\! b_{2} y\!+\!f\] \mathrm{d}s +\!\int_{t_0}^{T}\! \left(b_{3}  y \!+\!g \right) \mathrm{d} W(s)\!-\!u_T^\delta\Big\}^2\mathrm{d}x \nonumber\\ &+\!\alpha\mE\|y_{0}\|^2_{H^2(G)}, 
\qq\qq\qq  \forall y_0\in L^2_{\cF_0}(\Om;H^2(G)\cap H_0^1(G)),
\end{align}
where $y$ solves the following equation:
\begin{equation}\label{p2}
	\left\{\begin{array}{ll}
		\ds\mathrm{d} y-\sum_{i, j=1}^{n}(a^{i j} y_{i})_j \mathrm{d} t=\left( b_{1}\cd \nabla y +b_{2} y+f\right) \mathrm{d} t+\left(b_{3} y+g\right) \mathrm{d} W & \text { in } (0,T)\times G, \\ 
		\ns y=0 & \text { on } (0,T)\times \pa G, \\ 
		\ns y(0)=y_0, & \text { in } G.
	\end{array}\right.\end{equation}

\noindent{\bf Minimization problem.} Minimize the functional $J_\alpha(y_0)$ on the space $L^2_{\cF_{0}}(\Om;H^2(G)\cap H_0^1(G))$.

Since the functional  $J_\alpha(y_0)$ is coercive, convex and lower semi-continuous, it has a unique minimizer $\bar y_{0}^\d$.

\begin{theorem}[Convergence rate]
Assume that conditions (H1)--(H3) hold. Let $\l_0>1$ be the number in Theorem \ref{carleman} and let $\alpha=\alpha(\delta)=\delta^2$. Then there exists a number
\begin{equation}\label{620eq3}
\l_3=\l_3\(\sigma,\max_{i,j}\|a^{ij}\|_{L^\infty(G;\mathbb{R}^{n\times n})}, G\)\geq\l_0>1,
\end{equation}
such that if $\delta\in(0,\delta_0)$  and $\delta\in(0,1)$ is sufficient small such that
\begin{align}\label{620eq4-2}
 \l_3\leq\ln\[\(\ln\(\delta_0^{-\frac12}\)\)^{\frac{1}{\ln(T+1)}}\],\q
 \delta_0\leq e^{-3^{1-c}(\ln(\delta_0^{-1}))^c},
\end{align}
then the following convergence estimate of the Tikhonov method holds for every $\tau\in(0,T)$,
\begin{equation}
	\begin{array}{ll}\ds
\|y(\tau;\bar y_{0}^\d)-y(\tau;u_0)\|_{L^2_{\cF_{t_0}}(\Om;H^2(G)\cap H_0^1(G))}\\
\ns\ds\leq C_1\big(1+\|y_0\|_{L^2_{\cF_{t_0}}(\Om;H^2(G)\cap H_0^1(G))}\big)e^{-3^{-c}\big(\ln(\delta_0^{-1})\big)^c},
\end{array}
\end{equation}
where $c=c(\tau,T)=\frac{\ln(\tau+1)}{\ln(T+1)}\in(0,1)$.
\end{theorem}

\begin{proof}
For $y_0\in L^2_{\cF_{0}}(\Om;H^2(G)\cap H_0^1(G))$, denote by $y(\cd;y_{0})$ the solution to \eqref{p2} with $y(t_0)=y_{0}$. For simplicity of notations, we denote
$$\begin{aligned}
	\cA y_0\triangleq & y_0+ \int_0^T \[\sum_{i, j=1}^{n}(a^{i j}(s) y_{i}(s;y_0))_j\!+\!b_{1}(s)\cd \nabla y(s;y_0) \!+\! b_{2} y(s;y_0)\!+\!f(s)\]\text{d}s \\ 
	& +\int_0^T\int_G[b_3(s)y(s;y_0)+g(s)]\text{d}W(s), \quad \dbP\mbox{-a.s.} 
\end{aligned}$$ 
and
$$\begin{aligned}
	\cB y_0\triangleq & y_0+ \int_0^T \[\sum_{i, j=1}^{n}(a^{i j}(s) y_{i}(s;y_0))_j\!+\!b_{1}(s)\cd \nabla y(s;y_0) \!+\! b_{2}(s) y(s;y_0) \]\text{d}s \\
	& +\int_0^T\int_G b_3(s)y(s;y_0) \text{d}W(s), \quad \dbP\mbox{-a.s.} 
\end{aligned}$$ 

By computing G\^ateaux derivative of $J_\alpha(y_0)$, we get that
\begin{align}\label{620eq4}
J_\a'(y_0)\phi_0=2\mE\int_G \cA y_{0}\cB \phi_0 \mathrm{d}x-2\mE\int_G u_T(x)\cB \phi_0\mathrm{d}x +2\alpha\mE\langle y_{0},\phi_0\rangle_{H^2(G)},\\
\qq\q \forall \phi_0\in L^2_{\cF_{0}}(\Om;H^2(G)\cap H_0^1(G)).\nonumber
\end{align}
Particularly, for
the minimizer $\bar y_{0}^\d\in L^2_{\cF_{0}}(\Om;H^2(G)\cap H_0^1(G))$ of \eqref{616eq7}, we have
\begin{align}\label{620eq4-1}
\mE\int_G \cA \bar y_{0}^\d\cB\phi_0 \mathrm{d}x+ \alpha\mE\langle \bar y_{0}^\d,\phi_0\rangle_{H^2(G)} 
 =  \mE\int_G u_T^\d(x)\cB\phi_0\mathrm{d}x,\\
	\qquad\q  \forall \phi_0\in L^2_{\cF_{0}}(\Om;H^2(G)\cap H_0^1(G)).\nonumber
\end{align}
Noting that $u_0$ is the exact initial value, we have
\begin{align}\label{620eq4-3}
\mE\int_G \cA u_{0}\cB\phi_0 \mathrm{d}x+ \alpha\mE\langle u_{0},\phi_0\rangle_{H^2(G)} 
=  \mE\int_G u_T(x)\cB\phi_0\mathrm{d}x+ \alpha\mE\langle u_{0},\phi_0\rangle_{H^2(G)},\\
\qquad\q  \forall \phi_0\in L^2_{\cF_{0}}(\Om;H^2(G)\cap H_0^1(G)).\nonumber
\end{align}

From \eqref{620eq4-1} and \eqref{620eq4-3}, we obtain
\begin{align}\label{620eq5}
&\mE\int_G \big(\cA \bar y_{0}^\d-\cA u_{0}\big) \cB\phi_0 \mathrm{d}x+ \alpha\mE\langle \bar y_{0}^\d-u_{0},\phi_0\rangle_{H^2(G)}\\
\ns=&\int_G (u_T^\delta(x)-u_T(x)) \cB\phi_0 \mathrm{d}x-\alpha\mE\langle u_{0},\phi_0\rangle_{H^2(G)},\q  \forall \phi_0\in L^2_{\cF_{0}}(\Om;H^2(G)\cap H_0^1(G)).\nonumber
\end{align}
Set $\phi_0=\bar y_{0}^\d-u_0$ in \eqref{620eq5}. By using Cauchy-Schwarz inequality, we obtain
\begin{align}\label{620eq6}
&\int_G\big(\cA \bar y_{0}^\d-\cA u_0\big) ^2\mathrm{d}x+\alpha\mE\|\bar y_{0}^\d-u_0\|^2_{H^2(G)}\\
\leq & \frac12\int_G\big(\cA \bar y_{0}^\d-\cA u_0\big)^2\mathrm{d}x+\frac12\mE\| u_T^\delta-u_T\|_{H^2(G)}^2 + \frac12\alpha\|u_0\|_{H^2(G)}^2+\frac12\alpha\mE\|\bar y_{0}^\d-u_0\|^2_{H^2(G)}.\nonumber
\end{align}
Simplify \eqref{620eq6} and noting \eqref{620eq3}, we have
\begin{align}\label{620eq7}
\int_G\big(\cA \bar y_{0}^\d-\cA u_0\big)^2\mathrm{d}x+\alpha\mE\|\bar y_{0}^\d-u_0\|^2_{H^2(G)} \leq \delta^2+\alpha\|u_0\|_{H^2(G)}^2.
\end{align}
Since $\alpha=\alpha(\delta)=\delta^2$, we know from \eqref{620eq7} that
\begin{equation}\label{620eq8}
\mE\|u_{0}-\bar y_0^\d\|^2_{H^2(G)}\leq 1+\|u_0\|_{H^2(G)}^2.
\end{equation}

Let $\tilde u=y(\cd;u_{0})-y(\cd;\bar y_0^\d)$. Then $\tilde u$ solves the following equation
\begin{equation}\label{p2-1}
\left\{\begin{array}{ll}\ds\mathrm{d} \tilde{u}-\sum_{i, j=1}^{n}(a^{i j} \tilde{u}_{i})_j \mathrm{d} t=\left(b_{1}\cd\nabla \tilde{u}+b_{2} \tilde{u}\right) \mathrm{d} t+b_{3} \tilde{u} \mathrm{d} W(t)& \text { in } Q, \\
\ns \tilde{u}=0 & \text { on } \Sigma, \\
\ns \tilde{u}(0)=u_{0}-\bar y_0^\d, & \text { in } G.
\end{array}\right.\end{equation}
By \eqref{519eq3}, we get that
\begin{align}\label{620eq9}
&\l^2 \mathbb{E}\int_\e^T\int_G (t+1)^{(\l-2)}\f^2\tilde{u}^2\mathrm{d} t\mathrm{d}x+ \l\mathbb{E}\int_\e^T\int_G(t+1)^{-1}\f^2|\nabla \tilde{u}|^2\mathrm{d} t\mathrm{d}x\\
\ns\leq& C\[\mathbb{E}\l(T\!+\!1)^{\l-1}e^{2(T+1)^\l}\|\tilde{u}(T)\|^2_{L^2(G)}+\! \sum_{i,j=1}^n\!\|a^{ij}\|^2_{L_{\mathbb{F}}^{\infty} ( 0, T; L^{\infty}(G) )}\mathbb{E}\|\nabla(u_{0}\!-\!\bar y_0^\d)\|^2_{L^2(G)}\].\nonumber
\end{align}

Since 
\begin{align}\label{621eq1}
\mE\|\tilde{u}(T)\|^2_{L^2(G)}\leq \mE\|y(T;u_{0})-u_T^\delta\|^2_{L^2(G)}+\mE\|u_T^\delta-y(T; \bar y_0^\d)\|^2_{L^2(G)},
\end{align}
and due to $\bar y_{0}^\d$ is the minimizer of $J_\alpha(\cd)$, we know
\begin{equation}\label{621eq2}
\begin{array}{ll}\ds
\mE\|y(T;\bar y_{0}^\d)-u_T^\delta\|^2_{L^2(G)}\3n&\ds\leq J_\alpha(\bar y_{0}^\d)\leq J_\alpha(u_0)\\
\ns&\ds=\mE\|u_T^\delta-y(T; u_{0})\|^2_{L^2(G)}+\alpha\mE\|u_0\|^2_{H^2(G)}.
\end{array}
\end{equation}
Substituting \eqref{621eq2} into \eqref{621eq1}, we get
\begin{equation}\label{621eq3}
\begin{array}{ll}\ds
\mE\|\tilde{u}(T)\|^2_{L^2(G)}\3n&\ds\leq2\mE\|u_T^\delta-y(T; u_{0})\|^2_{L^2(G)}+\alpha\mE\|u_0\|^2_{H^2(G)}\\
\ns&\ds=2\delta^2+\alpha\mE\|u_0\|^2_{H^2(G)}.
\end{array}
\end{equation}

Combining \eqref{620eq9} with \eqref{621eq3}, we have
\begin{align*}\label{621eq4}
&\mathbb{E}\int_\e^T\int_G \l^2 (t+1)^{(\l-2)}e^{2(t+1)^\l}\tilde{u}^2\mathrm{d} t\mathrm{d}x+\mathbb{E}\int_\e^T\int_G \l(t+1)^{-1}e^{2(t+1)^\l}|\nabla \tilde{u}|^2\mathrm{d} t\mathrm{d}x\nonumber\\
\ns\leq&\frac12\mathbb{E}\l(T + 1)^{\l-1}\f(T)^2(2\delta^2 + \alpha\mE\|u_0\|^2_{H^2(G)})\\
& +\frac{e^2}2\sum_{i, j=1}^{n}\|a^{ij}(0)\|^2_{L^\infty_{\cF_0}(L^\infty(G))}\mathbb{E}\|\nabla(u_{0}\!-y_0)\|^2_{L^2(G)}.
\end{align*}

The following process is similar to the proof of Theorem \ref{reguest}.
 \end{proof}

\section{Numerical Approximation}

  Compared to deterministic parabolic equation, numerically solving the inverse problem of stochastic parabolic equation is more complicate and difficult. 
On one hand, the analytic solution of stochastic parabolic equation can not be explicitly expressed; on the other hand, the solution of stochastic parabolic equation is not differentiable with respect to the temporal  variable. Moreover, for many sharp method working well for deterministic problem, there are new essential difficulty. For example, when using the conjugate gradient method to the stochastic problem, the adjoint system is a backward stochastic parabolic equation. Thus, one has to solve a forward-backward stochastic parabolic equation numerically. The solution of this equation is three stochastic processes, whereas the last one with lower regularity, and thus is difficulty to be numerically solved. 
We  combine  conjugate gradient method and Picard type algorithm for forward-backward stochastic parabolic equation introduced in \cite{DP2016SISC} to solve \eqref{p1} numerically. 
To this end, we need to employ the adjoint equation of \eqref{p1}, which is a backward stochastic parabolic equation. As a result,  we assume that   $\left\{\mathcal{F}_{t}\right\}_{t \geqslant 0}$ is the natural filtration generated by $W(\cd)$.


\subsection{Conjugate gradient method  for the regularization method}
The conjugate gradient (CG) method is an iterative algorithm for the numerical solution of linear systems.  Thus can be used to numerically solve partial differential equations and unconstrained optimization problems such as energy minimization. For readers' convenience, we describe the conjugate gradient method for finding the solution $x$ which satisfies
$$
x=\operatorname{argmin}_{x\in \mathbb{R}^{n}} \left(\frac{1}{2} x^{\top} A x-b^{\top} x\right)=\operatorname{argmin}_{x\in \mathbb{R}^{n}}f(x).
$$
The numerical approximation of the above system given by the CG method after $k+1$ steps is
\begin{equation}\label{2.1-eq1}
x_{k+1}=x_0+\sum_{j=0}^{k}\alpha_j p_j=x_k+\alpha_k p_k.
\end{equation}
In \eqref{2.1-eq1},  $x_0$ is the initial guess of the solution, 
$$
p_{k}=r_{k}-\sum_{i=1}^{k-1} \frac{\left\langle r_{k}, A p_{i}\right\rangle}{\left\langle p_{i}, A p_{i}\right\rangle} p_{i}.
$$
for $r_k\triangleq b-A x_{k}$, and
the step size $\alpha_k$ can be chosen to minimize $f(x_k+\alpha_k p_k)$.

Now we turn to our problem. 
As we described in Section 4, our purpose is to minimize the functional
\begin{equation}\label{Tikhonov functional}
J_\alpha(y_0)= \mathbb{E} \int_G\big(\cA y_0-u_T^\delta(x)\big)^2\text{d}x +\alpha||y_0||_{H^2(G)}^2,\q y_0\in H^2(G)\cap H_0^1(G).
\end{equation}

Let $\{y_0^k\}_{k=1}^\infty$ be the approximation sequence generated by the conjugate gradient method to $\bar y_{0}$, i.e.,
\begin{equation} 
y_0^{k+1} = y_0^{k}+\beta_kd_k,\ k=0,1,2,\cdots
\label{uk}
\end{equation}
In \eqref{uk}, $k$ is the iteration index, $\beta_k$ and $d_k$ are the step size and descent direction in the $k$-th iteration given as follows:
\begin{equation}
\beta_k=-\frac{\ds\mathbb{E} \int_G(\cA y_0^k-u_T^\delta(x))\cA d_k\text{d}x +\alpha \langle y_0^k,d_k\rangle_{H^2(G)}}{\ds\mathbb{E} \int_G(\cA d_k)^2\text{d}x +\alpha |d_k|_{H^2(G)}^2}.
\label{betak}
\end{equation}
\begin{equation}
d_0=0,\q d_k=-J_\alpha'(y_0^k)+\gamma_kd_{k-1} \mbox{ for } k=1,2,\cds,
\label{dk}
\end{equation}
where $J_\alpha'$ is the G\^ateaux derivative of the functional \eqref{Tikhonov functional} and $\gamma_k$ is the conjugate coefficient given inductively: 
\begin{equation}
\gamma_0=0,\q \gamma_k=\[\mE\int_G(J_\alpha'(y_0^{k-1}))\text{d}x\]^{-1} \mE\int_G(J_\alpha'(y_0^k))^2\text{d}x, \q k=1,2,\cdots 
\label{gammak}
\end{equation}

By the above formulae, we should compute the sensitivity term $\cA d_k$ and the G\^ateaux derivative $J_\alpha'$. This is done in the new two subsections.

\subsection{The sensitivity problem and the adjoint problem}

In this subsection, we  compute the sensitivity term $\cA d_k$. For simplicity, we use $|\D v|_{L^2(G)}$ as the norm of $H^2(G)\cap H_0^1(G)$ for $v\in H^2(G)\cap H_0^1(G)$.

For any $v_0\in H^2(G)\cap H_0^1(G)$,
\begin{equation}
\begin{aligned}
  & J_\alpha(y_0+\delta v_0)-J_\alpha(y_0) \\
= & \mathbb{E}\big(2\langle \cA y_0-u_T^\delta,\cA\delta v_0\rangle _{L^2(G)}+ |\cA\delta v_0|_{L^2(G)}^2\big)+2\alpha\langle y_0,\delta v_0\rangle_{H^2(G)} +\alpha |\delta v_0|_{H^2(G)}^2\\
= & \mathbb{E}\big(2\langle \cA y_0-u_T^\delta,\cA\delta v_0\rangle _{L^2(G)}+ |\cA\delta v_0|_{L^2(G)}^2\big)+2\alpha\langle \D y_0,\delta \D v_0\rangle_{L^2(G)} +\alpha |\delta\D v_0|_{L^2(G)}^2.
\end{aligned}
\end{equation}
Since $\cA$ is a linearly continuous operator, we have
\begin{equation}
\begin{aligned}
&\lim_{\delta\to0}\frac{J_\alpha(y_0+\delta v_0)-J_\alpha(y_0)}{\delta} \\
=&  2\mathbb{E}\langle \cA y_0-u_T^\delta,\cA v_0\rangle _{L^2(G)} +2\alpha\langle \D y_0,\D v_0\rangle_{H^2(G)}\\
=& 2\mathbb{E}\langle (-\D)^{-1}(\cA y_0-u_T^\delta),(-\D)^{-1}\cA v_0\rangle _{H^2(G)} +2\alpha\langle  y_0, v_0\rangle_{H^2(G)}.
\end{aligned}
\end{equation}
Then the G\^{a}teaux derivative of $J$ is
\begin{equation}
J_\alpha'(y_0)=(-\D)^{-2}\cA^*(\cA y_0-u_T^\delta)+\alpha y_0.
\label{J'}
\end{equation}

Denote $V=y(\cd,\bar y_{0})-y(\cd,y_0)$, then the following sensitivity problem is obtained
\begin{equation}
\begin{cases}
\text{d}V-\sum\limits_{i,j=1}^n(a^{ij}V_i)_j\text{d}t=\left(b_1\cd\nabla V+b_2V\right)\text{d}t+b_3VdW(t) & \text{in}\ Q,\\
V=0 & \text{on}\ \Sigma, \\
V(0)=\bar y_{0}-y_0 & \text{in}\ G.
\end{cases}
\label{sensitivity problem}
\end{equation}
The adjoint problem of (\ref{sensitivity problem}) is
\begin{equation}
\begin{cases}
\text{d}Y=\[-\sum\limits_{i,j=1}^n(a^{ij}Y_i)_j-\nabla\cdot(b_1 Y)-b_2 Y-b_3  Z \]\text{d}t+ZdW(t) & \text{in}\ Q,\\
Y=0 & \text{on}\ \Sigma, \\
Y(T)=\cA y_0-u_T^\delta & \text{in}\ G.
\end{cases}
\label{adjoint problem}
\end{equation}
The existence and uniqueness of the strong solution of this adjoint problem is proved in \cite{DT2012}.


We summarise the conjugate gradient method to solve the minimization problem (\ref{Tikhonov functional}) in Algorithm 1.

\begin{table}[htbp]
\centering
\begin{tabular}{ll}
\toprule
&\textbf{Algorithm 1} Conjugate gradient algorithm for the inverse problem \bf{IPD}\\
\midrule
1: & Choose an initial guess $y_0^0$. Set $k=0$. \\
2: & Solve the initial boundary problem (\ref{p1}) with $y_0^k$, \\
   & and determine the residual $r_k=y^k(T)-u_T^\delta$. \\
3: & Solve the adjoint problem (\ref{adjoint problem}) and determine $J_\alpha'$ by (\ref{J'}). \\
4: & Calculate the conjugate coefficient $\gamma_k$ by \eqref{gammak} and the descent direction $d_k$ by \eqref{dk}. \\
5: & Solve the sensitivity problem (\ref{sensitivity problem}) for $Ad_k$ with $V(0)=d_k$. \\
6: & Calculate the stepsize $\beta_k$ by (\ref{betak}). \\
7: & Compute a new estimate, $y_0^{k+1}$, with (\ref{uk}). \\
8: & Interrupt the iterative procedure if the stopping criterion is satisfied.\\ 
   & Otherwise, increase k by 1 and go back to Step 2.\\
\bottomrule
\end{tabular}
\end{table}

\subsection{Picard type algorithm for forward-backward stochastic parabolic equation}


In this subsection, borrowing some idea in \cite{DP2016SISC}, we present fully implementable algorithms to simulate the equations (\ref{p1}), (\ref{sensitivity problem}) and (\ref{adjoint problem}). 

For simplicity of notations, we suppose $(a^{ij})_{1\leq i,j\leq n}$ be the identity matrix and $f=g=0$ in (\ref{FBSPDE}) throughout this subsection. 

We write (\ref{p1}) and (\ref{adjoint problem}) together as follow and called them the forward-backward stochastic heat equation (FBSPDE)
\begin{equation}
\begin{cases}
\text{d}U=\big(\D U + b_1\cd\nabla U +b_2U\big)\text{d}t+ b_3U dW(t) & \text{in}\ Q,\\
\text{d}Y=-\big[\D Y +\div(b_1 Y)+b_2 Y+b_3 Z\big]\text{d}t+ZdW(t) & \text{in}\ Q,\\
U(0)=u_0 & \text{in}\ G,\\
Y(T)=\cA y_0-u_T^\delta & \text{in}\ G.
\end{cases}
\label{FBSPDE}
\end{equation}
%


In order to discretize above equations, we introduce some notation here. Let $\mathcal{M}_h$ be a regular mesh of $G\subset\mathbb{R}^n$ into element domains $K$ with a maximum mesh size $h\triangleq\max\{diam(K)|K\in\mathcal{M}_h\}$. For each $K\in\mathcal{M}_h$, let $\mathcal{P}_j(K)$ denote the set of all polynomials of degree less than or equal to $j$ on $K$, and we define the finite element space $\mathbb{U}_h\subset H_0^{1}$ by
\begin{equation*}
	\mathbb{U}_h\triangleq\{\phi\in C^k(\bar{G})|\phi_{|K}\in\mathcal{P}_j(K),\ \forall K\in\mathcal{M}_h\}.
\end{equation*}
The $\mathbb{L}^2$-projection $\Pi_h: \mathbb{L}^2\to \mathbb{U}_h$ is defined by $(\Pi_h\xi-\xi, \phi_h)=0$ for all $\phi_h\in\mathbb{U}_h$.

The spatial discretization of (\ref{FBSPDE}) is as follows:

For all $t\in[0,T]$, there holds $\mathbb{P}$-a.s.
\begin{equation}\label{2.3-eq1}
\begin{aligned}
&\int_G U_h(t) \phi_h\text{d}x+ \int_0^t\int_G \nabla U_h \cd \nabla\phi_h \text{d}x\text{d}s\\
=&\int_G u_0 \phi_h\text{d}x-\int_0^t\int_G [U_h,\div(b_1\phi_h)- b_2U_h \phi_h ]\text{d}x\text{d}s+\int_0^t\int_G b_3U_h \phi_h \text{d}x\text{d}W(s),
\end{aligned}
\end{equation}
and
\begin{equation}\label{2.3-eq2}
\begin{aligned}
\int_G Y_h(t) \phi_h \text{d}x= & \int_G (\cA y_0-u_T^\delta)\phi_h \text{d}x- \int_t^T\int_G (\nabla Y_h\cd\nabla\phi_h + Y_hb_1\cd\nabla\phi_h ) \text{d}x\text{d}s\\
& +\int_t^T\int_G (b_2  Y_h \phi_h + b_3  Z_h \phi_h)\text{d}x\text{d}s-\int_t^T\int_G  Z_h \phi_h\text{d}x\text{d}W(s).
\end{aligned}
\end{equation}
For every fixed $h>0$, there exsit a unique  solution $(U_h,Y_h,Z_h)\in L_\dbF^2(\Omega,C([0,T];\mathbb{U}_h))\times L_\dbF^2(\Omega,C([0,T];\mathbb{U}_h)) \times L_\dbF^2(\Omega,L^2([0,T];\mathbb{U}_h))$ to \eqref{2.3-eq1} and \eqref{2.3-eq2}.

Denote the time discretization of $(U_h,Y_h,Z_h)$  by $\{(U_h^m(t),Y_h^m(t),Z_h^m(t))|m=0,\cdots,\dbM\}$. Let $k=t_{m+1}-t_m$ be the uniform time step for a net $\{t_m\}_{m=0}^\mathbb{M}$ which covers $[0,T]$. Let $\Delta_mW=W(t_m)-W(t_{m-1})$. Then we can simulate $\{(U_h^m(t),Y_h^m(t),Z_h^m(t))|m=0,\cdots,\dbM\}$ as follow:

(i) Simulate $U_h^0=\Pi_h(y_0)$.

(ii) For each $m=0,\cdots,\dbM-1$, simulate $U_h^m$ such that for each $ \phi_h\in\mathbb{U}_h$,
$$
\begin{aligned}
&\int_G U_h^{m+1} \phi_h \text{d}x +k\int_G \nabla U^{m+1}_h\cd\nabla\phi_h \text{d}x\\
= & \int_G U_h^m \phi_h \text{d}x -k\int_G \big[U_h^m \div(b_1\phi_h)- b_2U_h^m \phi_h\big] \text{d}x+  \D_{m+1}W \int_G b_3U_h^m \phi_h \text{d}x.
\end{aligned}
$$

(iii) Simulate $Y_h^\dbM=U_h^{\mathbb{M}}-\Pi_h(u_h^{\dbM})$.

(iv) For each $m=\dbM-1,\cdots,0$, simulate $Y_h^m$ and $Z_h^m$ such that for each $  \phi_h\in\mathbb{U}_h$,
$$
\int_G Z_h^m \phi_h \text{d}x=\frac{1}{k}\mathbb{E}\left[\D_{m+1}W \int_G Y_h^{m+1} \phi_h \text{d}x\Big|\cF_{t_m}\right],
$$
and
$$
\begin{aligned}
& \int_G Y_h^m \phi_h\text{d}x + k\int_G \nabla Y^{m}_h \cd\nabla \phi_h\text{d}x+k \int_G (Y_h^m b_1 \cd \nabla\phi_h - b_2 Y_h^m \phi_h)\text{d}x\\
&=\mathbb{E}\left[\int_G Y_h^{m+1} \phi_h\text{d}x\Big|\cF_{t_m}\right]+k\int_G b_3Z^m_h \phi_h\text{d}x.
\end{aligned}
$$

For $l=1,\dots,L$, let $\phi_h^l\in\mathbb{U}_h$ be basis functions of $\mathbb{U}_h$. Let $U_h(x,t)=\sum\limits_{l=1}^L{u_h^l(t)\phi_h^l(x)}$, $Y_h(x,t)=\sum\limits_{l=1}^L{y_h^l(t)\phi_h^l(x)}$ and $Z_h(x,t)=\sum\limits_{l=1}^L{z_h^l(t)\phi_h^l(x)}$ with coefficient vectors $\vec{U}_h, \vec{Y}_h,$ $ \vec{Z}_h\in\mathbb{R}^L$. Denote by $\St$ the stiffness matrix  consisting of entries $\int_G \nabla\phi_h^l\cd\nabla\phi_h^w \text{d}x$, where $\phi_h^l,\phi_h^w\in\mathbb{U}_h$ are basis functions of $\mathbb{U}_h$; by $\Mg$   the matirx consisting of entries $\int_G \phi_h^l b_1 \cd\nabla\phi_h^w \text{d}x$; by $\Md$  the matrix consisting of entries   $\int_G \nabla\phi_h^l\cd \nabla\cdot(b_1\phi_h^w)\text{d}x$ and by \textbf{Mass} the mass matrices consisting of entries $\int_G \phi_h^l \phi_h^w\text{d}x$. Then the semi-discretization above can be restated as an algebraic problem:
\begin{equation}
\begin{aligned}
(\Ma+k\ \St)\ \vec{U}_h^{m+1}= & [-k\ \Md+(b_2k+1)\ \Ma]\ \vec{U}_h^m +b_3\ \Ma\ \vec{U}_h^j\D_{m+1}W,
\end{aligned}
\label{iter_U}
\end{equation}
\begin{equation}
\begin{aligned}
\Ma\ \vec{Z}_h^m=&\frac{1}{k}\mathbb{E}\big[\D_{m+1}W\ \Ma\ \vec{Y}_h^{m+1}|\mathcal{F}_{t_m}\big],
\end{aligned}
\label{iter_Z}
\end{equation}
and
\begin{equation}
\begin{aligned}
[(1-b_2 )\ \Ma+k\ \St+k\ \Mg]\ \vec{Y}_h^{m}=\mathbb{E}\big[\Ma\ \vec{Y}_h^{m+1}|\mathcal{F}_{t_m}\big]+kb_3  \Ma\ \vec{Z}_h^m,
\end{aligned}
\label{iter_Y}
\end{equation}
by
\begin{equation}
\vec{Y}_h^m=A_{Y^{m}}\vec{U}_h^{m}+\vec{V}^{m},
\label{expressed_Y}
\end{equation}
with (deterministic) $A_{Y^m}\in\mathbb{R}^{L\times L}$ and $\vec{V}^m\in\mathbb{R}^L$ such that $A_{Y^\dbM}$ is the $L$-dimensional identity matrix  and $\vec{V}^\dbM=-\vec{u}_T^\delta$ respectively. The purpose of introducing expectations into \eqref{iter_Z} and \eqref{iter_Y} is to eliminate the non-uniqueness of $(\vec{Y}_h^m,\vec{Z}_h^m)$. For the sake of brevity, we denote $\Au=-k\ \Md+(b_2k+1)\ \Ma$ and $\Ay=(1-b_2)\ \Ma+k\ \St+k\ \Mg$ below. 

Next, we will derive the recursive form of $A_{Y^m}$ and $\vec{V}^m$ ($m=\dbM-1,\dots,0$). Using (\ref{iter_U}) and (\ref{expressed_Y}) we can compute $\vec{Z}_h^{m}$ as follows:
\begin{equation}\label{2.3-eq4}
\begin{aligned}
\vec{Z}_h^m= & \frac{1}{k}\mathbb{E}\big[\D_{m+1}W\ A_{Y^{m+1}}\ \vec{U}_h^{m+1}+\D_{m+1}W\ \vec{V}_{Y^{m+1}}|\mathcal{F}_{t_m}\big] \\
= & \frac{1}{k}\mathbb{E}\big[\D_{m+1}W\ A_{Y^{m+1}}(\Ma+k\ \St)^{-1}(\Au\ \vec{U}_h^m +b_3\ \Ma\ \vec{U}_h^m\D_{m+1}W)|\mathcal{F}_{t_m}\big]\\
= & A_{Y^{m+1}}(\Ma+k\ \St)^{-1}b_3\ \Ma\ \vec{U}_h^m.
\end{aligned}
\end{equation}
Combining \eqref{2.3-eq4} and (\ref{iter_Y}), we see %
\begin{equation}
\begin{aligned}
\vec{Y}_h^m= & \mathbb{E}\big[\Ay^{-1}\Ma\ (A_{Y^{m+1}}\vec{U}_h^{m+1}+\vec{V}^{m+1})|\mathcal{F}_{t_m}\big] \\
& +kb_3 \Ay^{-1}\Ma\ A_{Y^{m+1}}(\Ma+k\St)^{-1}b_3\ \Ma\ \vec{U}_h^m \\
\triangleq & A_{Y^{m}}\vec{U}_h^{m}+\vec{V}^{m}.
\end{aligned}
\end{equation}
By (\ref{iter_U}), we find that
\begin{equation*}
\begin{aligned}
\vec{Y}_h^m= & \big[\Ay^{-1}\Ma\ A_{Y^{m+1}}(\Ma+k\ \St)^{-1}\Au \\
& +kb_3 \Ay^{-1}\Ma A_{Y^{m+1}}(\Ma+k\St)^{-1}b_3\Ma\big]\vec{U}_h^m +\Ay^{-1}\Ma\ \vec{V}^{m+1}.
\end{aligned}
\end{equation*}
After simple calculation and comparison, we can determine $A_{Y^m}$ and $\vec{V}^m$ ($m=\dbM-1,\dots,0$) by
\begin{equation}
\begin{aligned}
A_{Y^m}= & \Ay^{-1}\Ma\ A_{Y^{m+1}}(\Ma+k\ \St)^{-1}\Au \\
& +kb_3 \Ay^{-1}\Ma A_{Y^{m+1}}(\Ma+k\St)^{-1}b_3\Ma,
\end{aligned}
\label{AYj}
\end{equation}
and
\begin{equation}
\begin{aligned}
\vec{V}^{m}= & \Ay^{-1}\Ma\ \vec{V}^{m+1}.
\end{aligned}
\label{Vj}
\end{equation}



\subsection{Numerical Examples}

In this subsection, we assume that the data $u_T^\delta=u_T+\delta\Vert{u_T}\Vert_\infty*2(rand(size(u_T))-0.5)$, where $\delta$ is the tolerated noise level and $2(rand(size(u_T))-0.5)$ generates random numbers uniformly distributed between $[-1,1]$.

In order to compare the numerical accuracy, we choose some extra test points to compute the root mean square error (RMSE):
\begin{equation}
RMSE=\Vert{u(\cdot,0)-u_c(\cdot,0)}\Vert_{l^2}=\sqrt{\frac{1}{N}\sum\limits_{i=1}^N(u(x_i,0)-u_c(x_i,0))^2},
\end{equation}
\begin{equation}
rmse=\frac{\Vert{u(\cdot,0)-u_c(\cdot,0)}\Vert_{l^2}}{\Vert{u_c(\cdot,0)}\Vert_{l^2}}=\sqrt{\frac{\sum\limits_{i=1}^N(u(x_i,0)-u_c(x_i,0))^2}{\sum\limits_{i=1}^Nu(x_i,0)^2}},
\end{equation}
where $u$ and $u_c$ are the exact and computational solutions of the problem, respectively. Here, $\{x_i\}_{i=1}^N$ is the vertices of our mesh, which are uniformly distributed in $G$.

Consider the following two examples:
\begin{example}\label{ex1}
Let $G=[0,1]$ and $T=1$. The initial value $u(x,0)=4x(1-x)$. The coefficients are $b_1=0$,  $b_2=0$ and $b_3=0.1$.
\end{example}
\begin{example}\label{ex2}
Let $G=[0,1]$ and $T=1$.  The initial value 
\begin{equation*}
u(x,0)=
\begin{cases}
2x, & x\in[0,0.5],\\
2-2x, & x\in(0.5,1].
\end{cases}
\end{equation*}
The coefficients are $b_1=0$,  $b_2=0$ and $b_3=0.1$.
\end{example}
Let the spatial size $h=1/20$,  and the temporal stepsize $k\le h^2$. 
In the computation, we set initial guess $u_0^0=0$ in Examples \ref{ex1} and \ref{ex2} and we simulate  at different noise levels $\delta=0,0.004,0.02,0.05,0.1$.  
Similar to the backward problem of deterministic  parabolic problem, the solution at $t=0$ is more difficult to retrieve than at $t\in(0,T)$. Thus, we only illustrate the numerical results at $t=0$ for Example \ref{ex1}. 

\begin{figure}[htbp]
\centering
\subfigure{
\includegraphics[width=6.5cm]{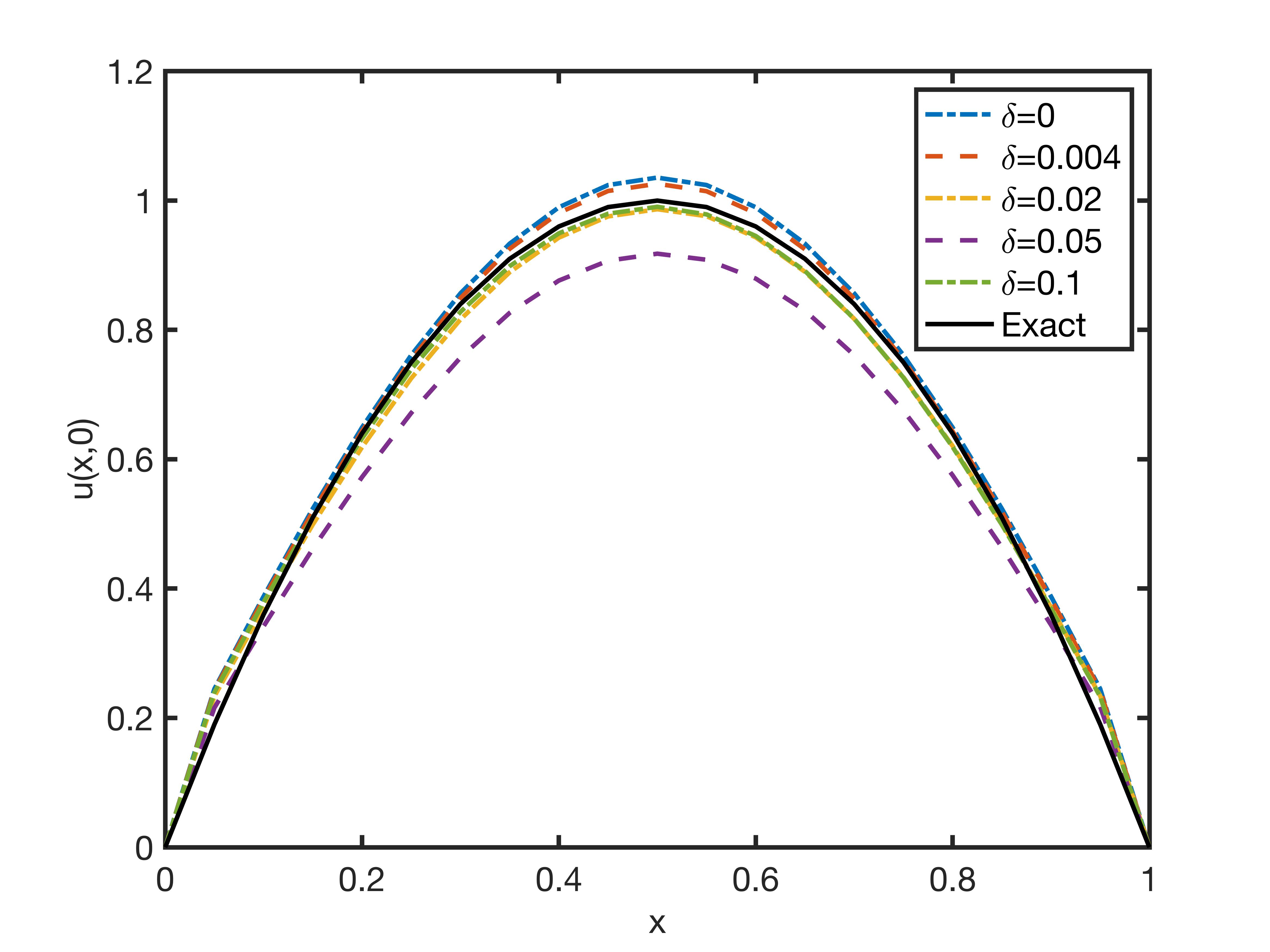}
}
\quad
\subfigure{
\includegraphics[width=6.5cm]{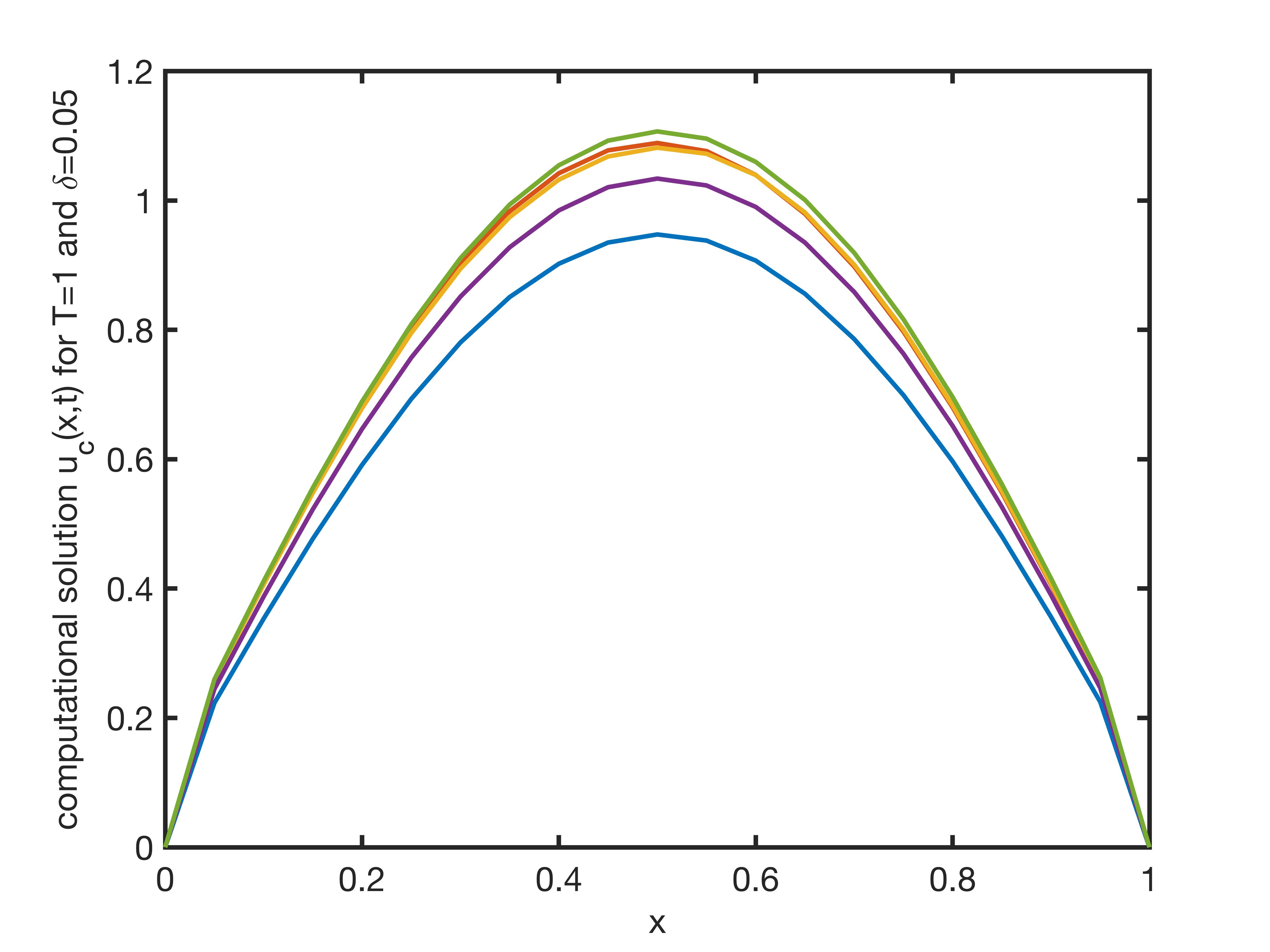}
}
\caption{(a) exact solution $u(x,0)$ and the computational approximations with different noise levels  $\delta=0, 0.004, 0.02, 0.05, 0.1$; (b) computational solution with $T=1$ and $\delta=0.05$.}
\label{results_1D}
\end{figure}

Noting that, due to effect of the stochastic term, the numerical solutions vary even at the same noisy level. This can be illustrated in Fig.\ref{results_1D}(b). 

\begin{table}[htbp]
\centering
\setlength{\tabcolsep}{6mm}
\begin{tabular}{cccccc}
\toprule
$T$ & $\delta=0$&$\delta=4\times10^{-3}$ & $\delta=2\times10^{-2}$ & $\delta=5\times10^{-2}$ & $\delta=10^{-1}$\\
\midrule
	$0.5$ & $0.1317$ & $0.1366$ & $0.1365$ & $0.1361$ & $0.1332$ \\
				& $0.0767$ & $0.0769$ & $0.0775$ & $0.0769$ & $0.0750$ \\
				\\
	$1$   & $0.1482$ & $0.1445$ & $0.1515$ & $0.1468$ & $0.1465$ \\
				 & $0.0830$ & $0.0815$ & $0.0854$ & $0.0817$ & $0.0829$ \\
				\\
	$1.2$  & $0.1558$ & $0.1492$ & $0.1549$ & $0.1569$ & $0.1534$ \\
				 & $0.0850$ & $0.0846$ & $0.0874$ & $0.0865$ & $0.0889$ \\
				\\
	$1.5$ & $0.1623$ & $0.1626$ & $0.1670$ & $0.1610$ & $0.1683$ \\
				 & $0.0907$ & $0.0928$ & $0.0878$ & $0.0921$ & $0.0981$ \\
\bottomrule
\end{tabular}
\caption{RMSE (first row) and rmse (second row) for $T=0.5,1.0,1.2,1.5$ for Example \ref{ex1} with different noise level $\delta$.}\label{RMSE}
\end{table}

%

We ran Example \ref{ex1} and \ref{ex2} 1000 times to evaluate the effect of the calculation. 
Table \ref{RMSE} shows the root mean square errors $RMSE$ and relative error $rmse$ of the numerical approximations for Example \ref{ex1} with $h=1/20$ and $k=1/800$. 
We set the measurement errors $\delta=0,0.004,0.02,0.05,0.1$ corresponding, respectively, to measurement errors of $0\%, 1\%, 5\%, 13\%$ and $25\%$ with respect to the largest value of $r^\delta$.  It can be seen that with the increase of noise, the calculation error does not increase sharply, which indicates that our algorithm has good robustness to noise.

\begin{figure}[htbp]
\centering
\subfigure{
\includegraphics[width=0.31\linewidth]{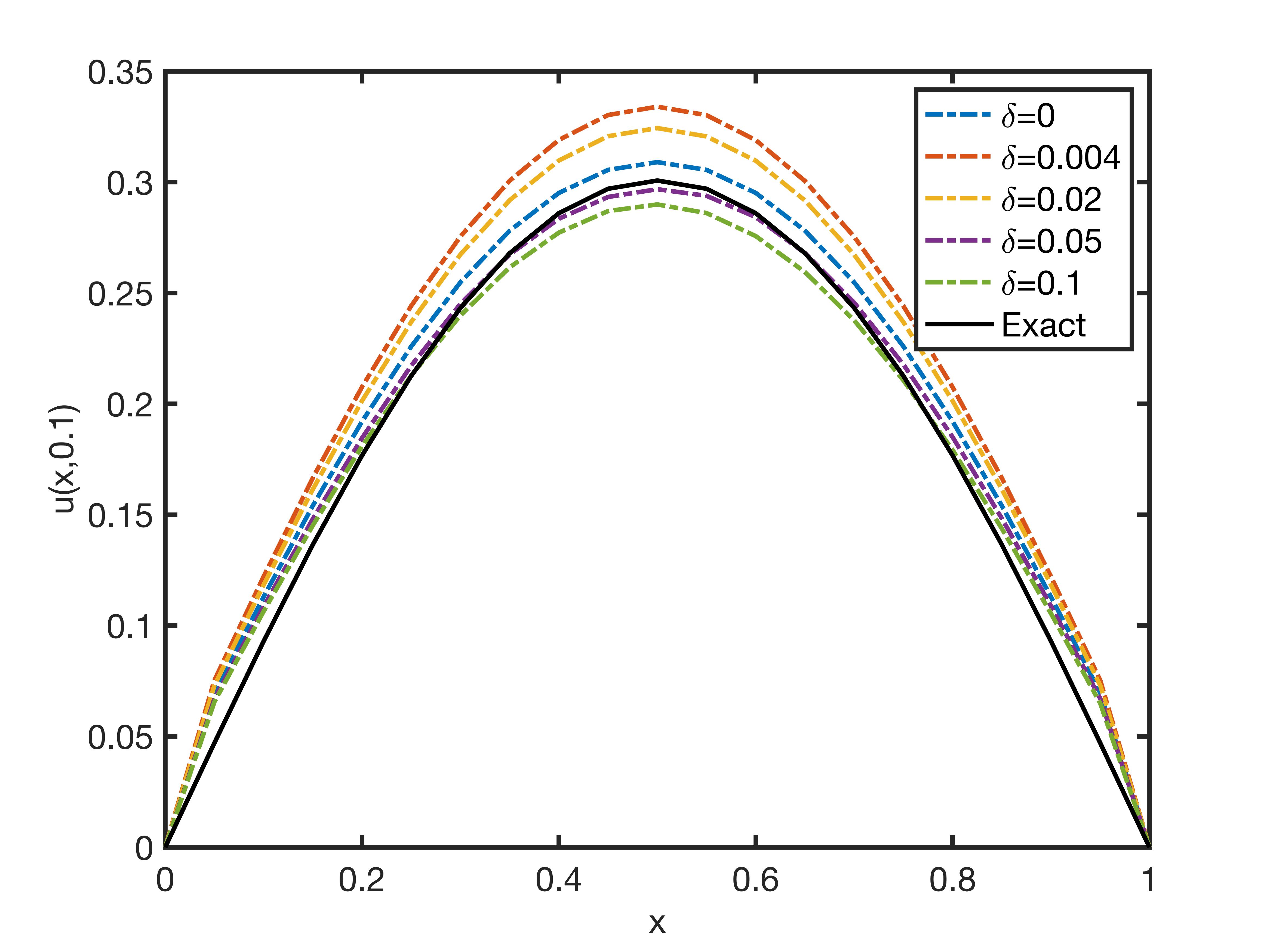}
}
\subfigure{
\includegraphics[width=0.31\linewidth]{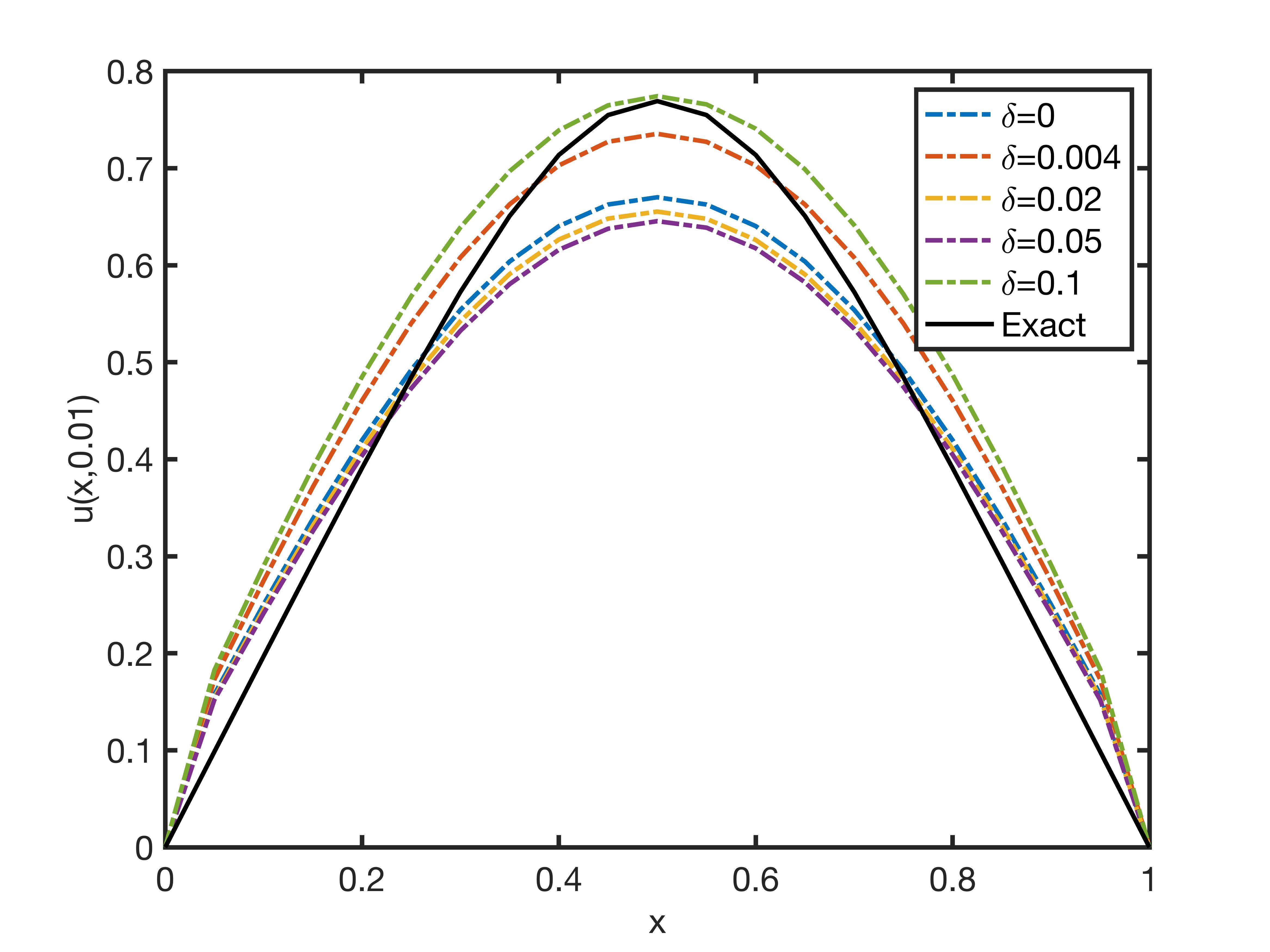}
}
\subfigure{
\includegraphics[width=0.31\linewidth]{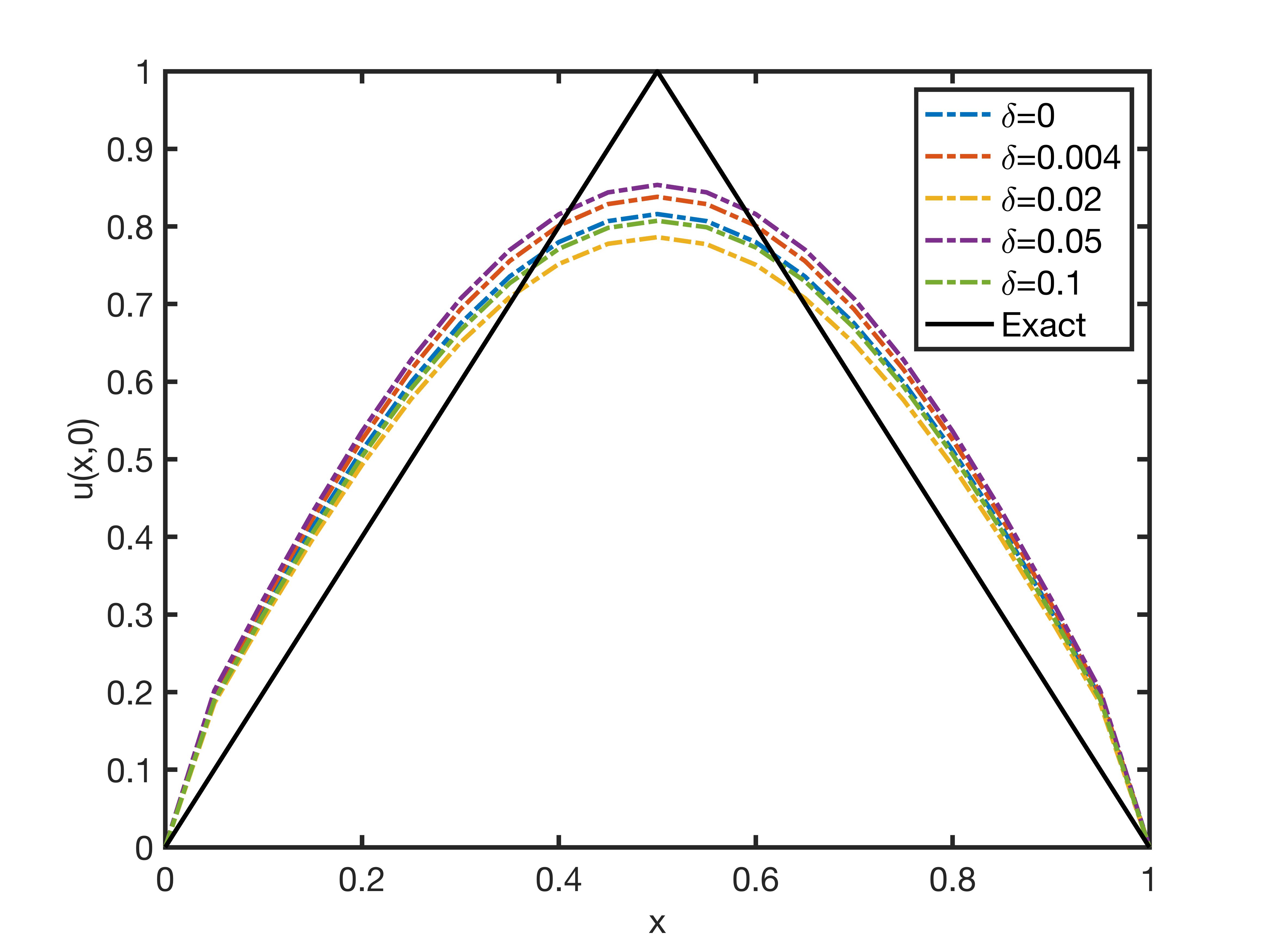}
}
\caption{computational solution at different time $t=0.1,0.01,0$.}
\end{figure}

%

We also verify the numerical algorithm by using an example in two dimensional spatial space. Figure \ref{results} shows that our method also work well for the case in a two dimensional spatial domain.
\begin{example}Let $G=[-1,1]\times [-1,1]$.  The initial value $u(\textbf{x},0)=\sin(\pi x_1)\cdot\sin(\pi x_2)$, where $\textbf{x}=(x_1,x_2)$.
\end{example}
\begin{figure}[htbp]
\centering
\subfigure[]{
\includegraphics[width=0.31\linewidth]{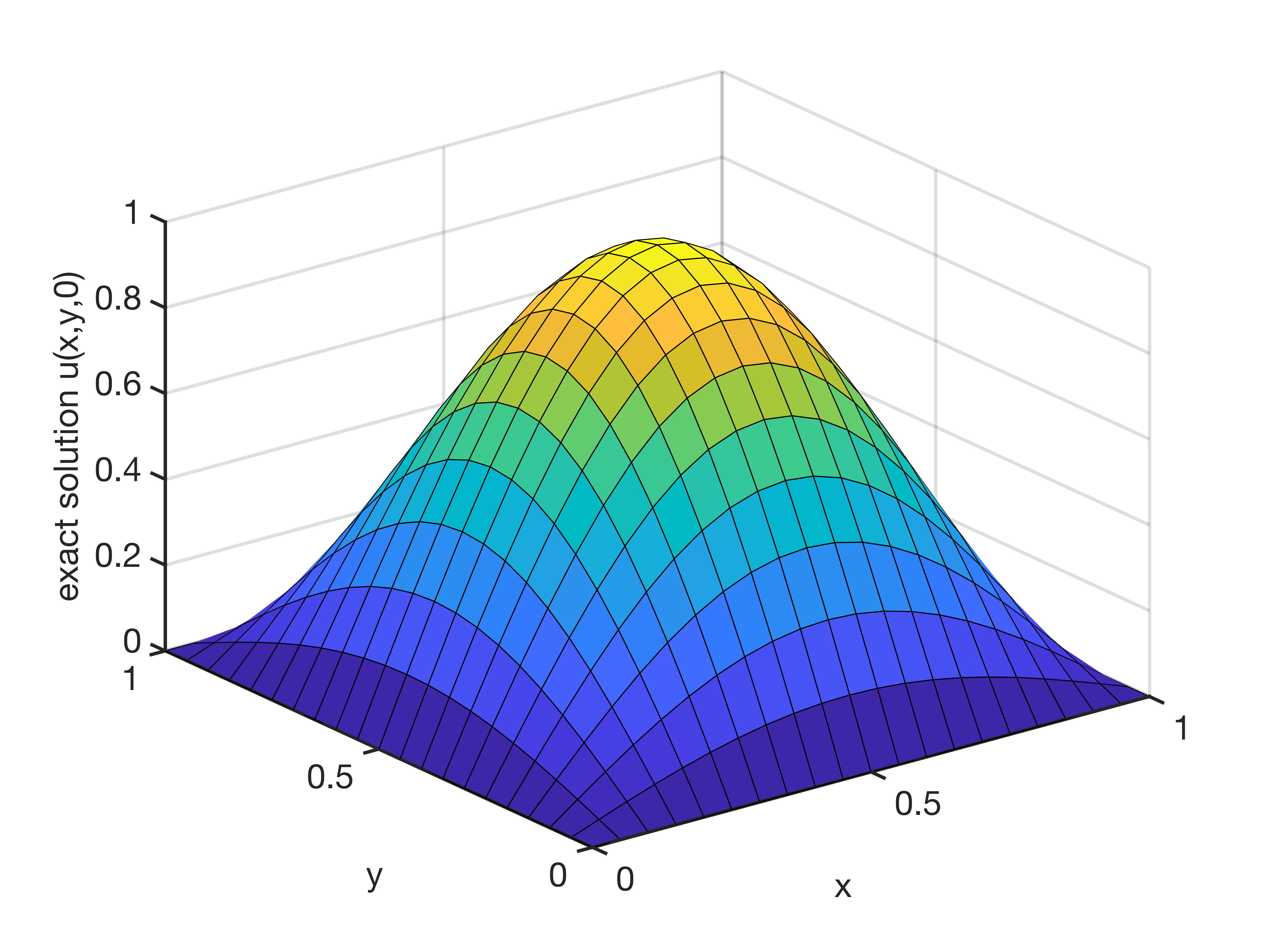}
}
\subfigure[]{
\includegraphics[width=0.31\linewidth]{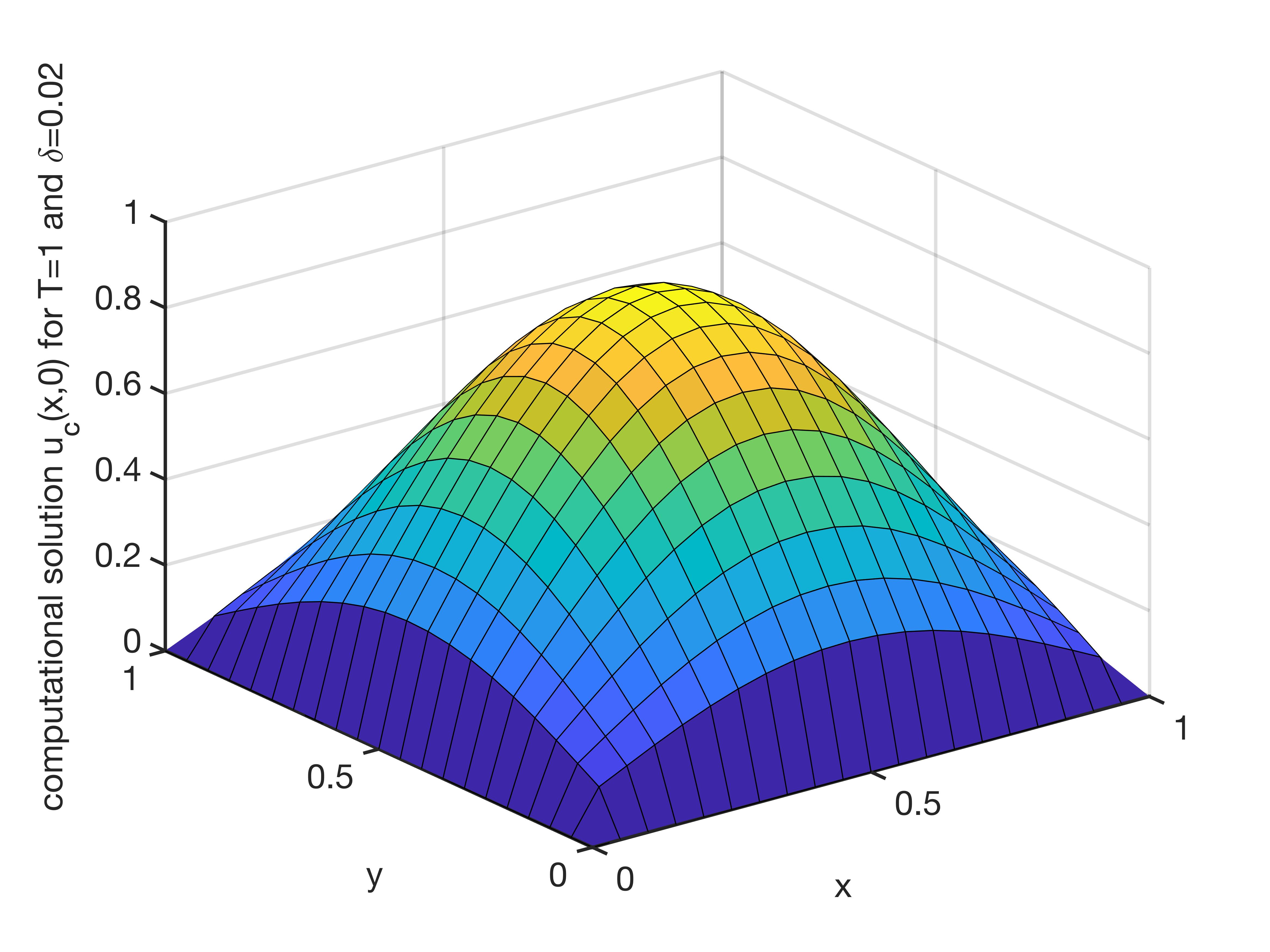}
}
\subfigure[]{
\includegraphics[width=0.31\linewidth]{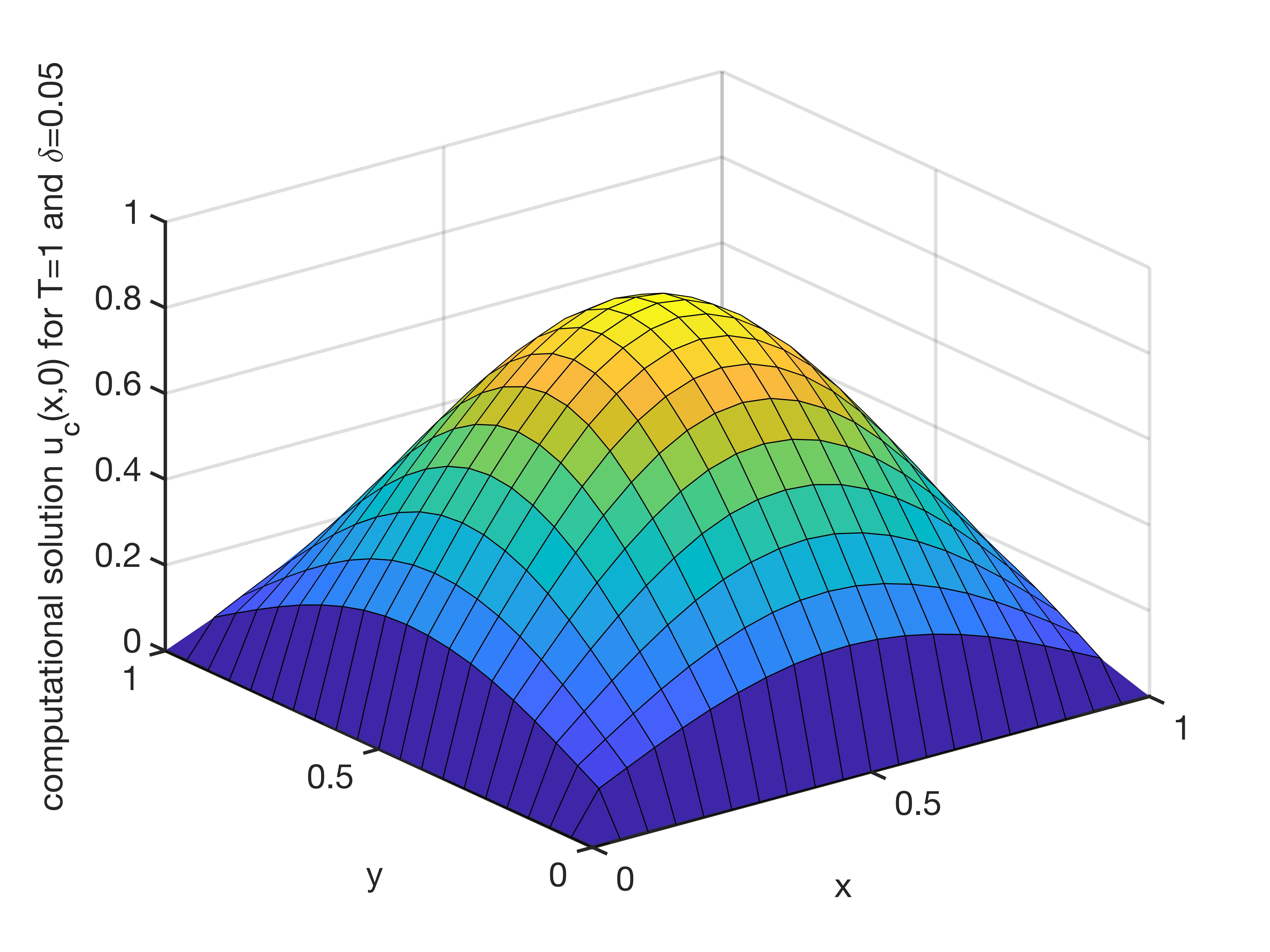}
}
\\
\subfigure[]{
\includegraphics[width=0.31\linewidth]{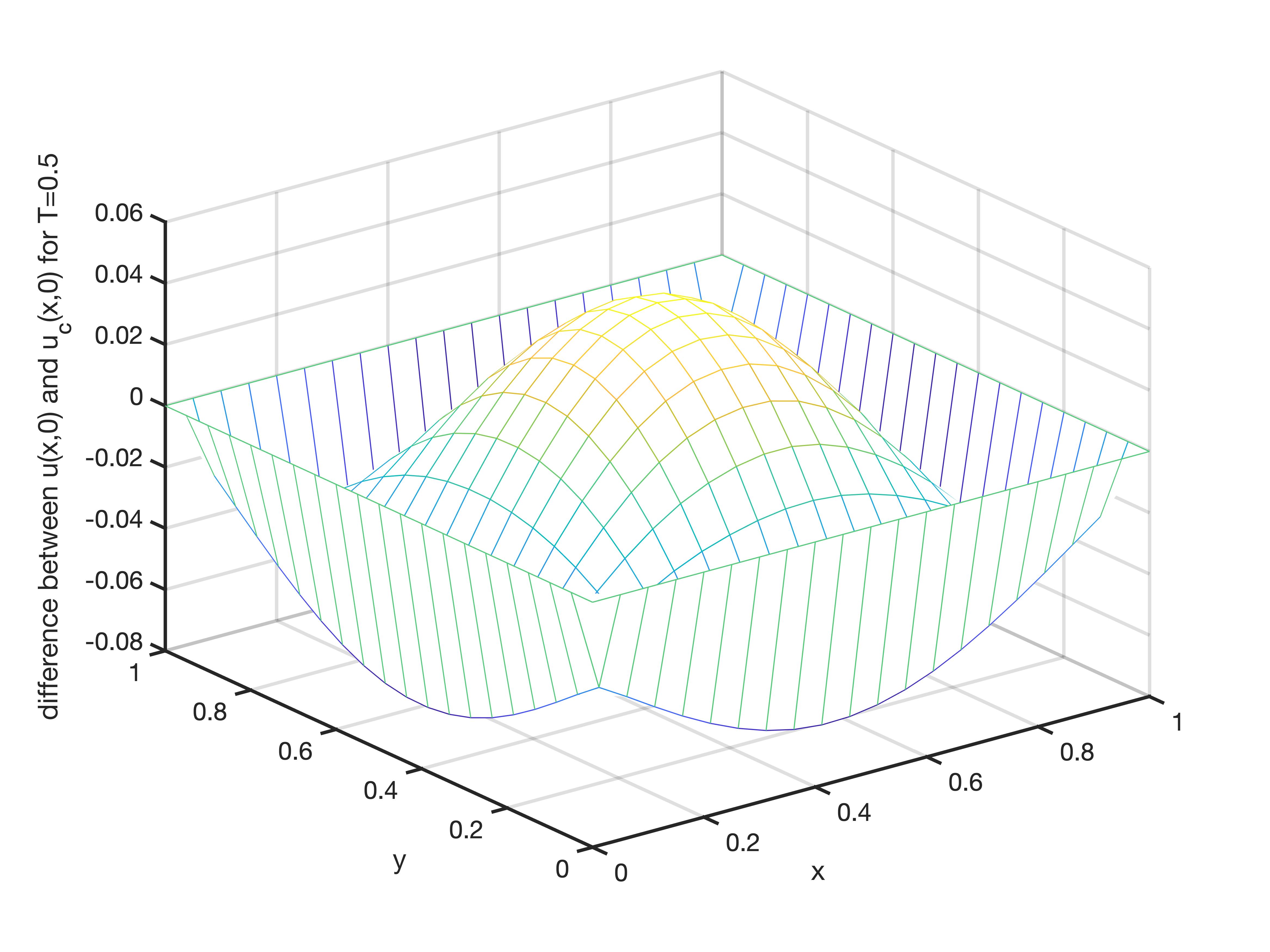}
}
\subfigure[]{
\includegraphics[width=0.31\linewidth]{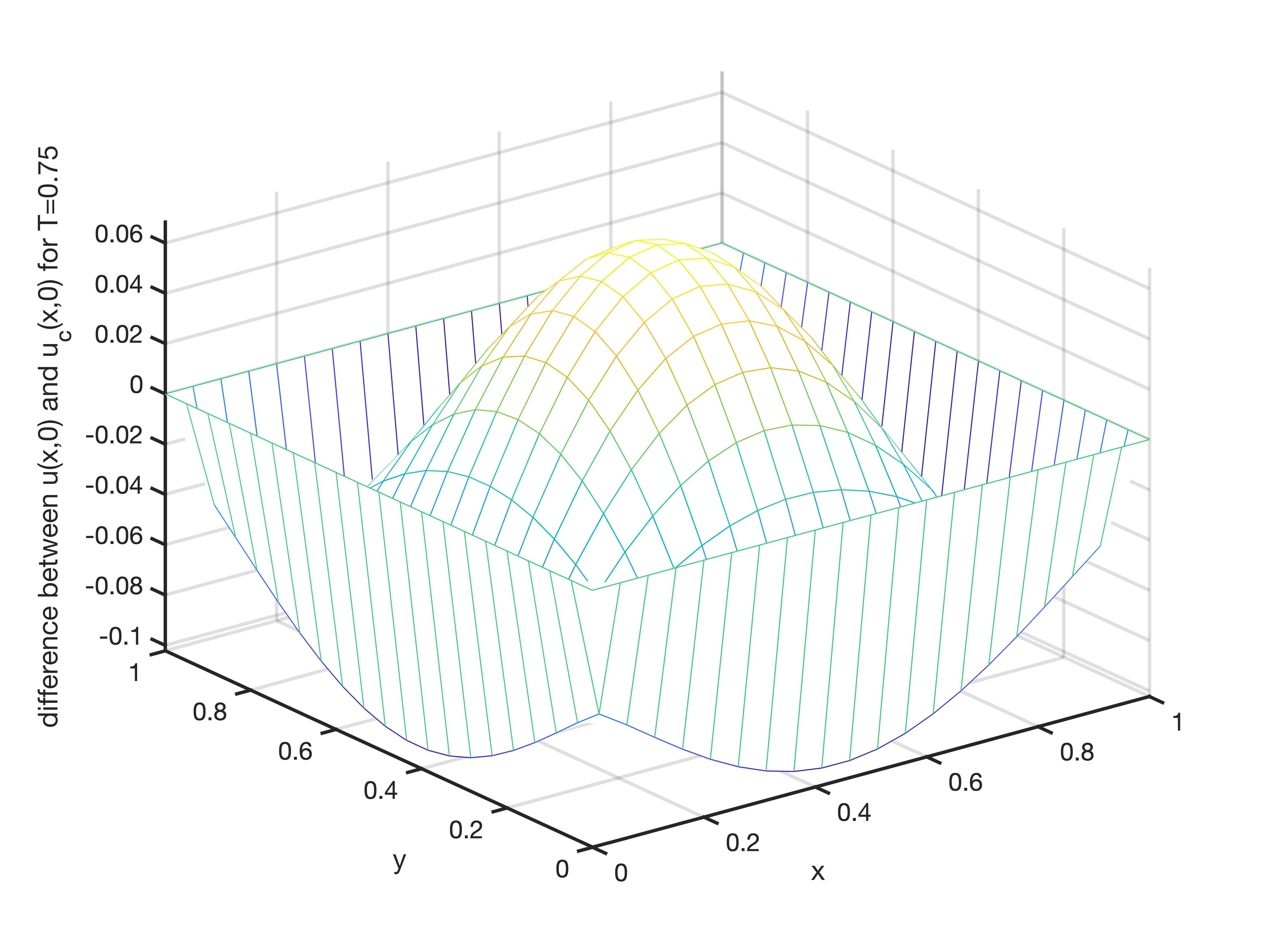}
}
\subfigure[]{
\includegraphics[width=0.31\linewidth]{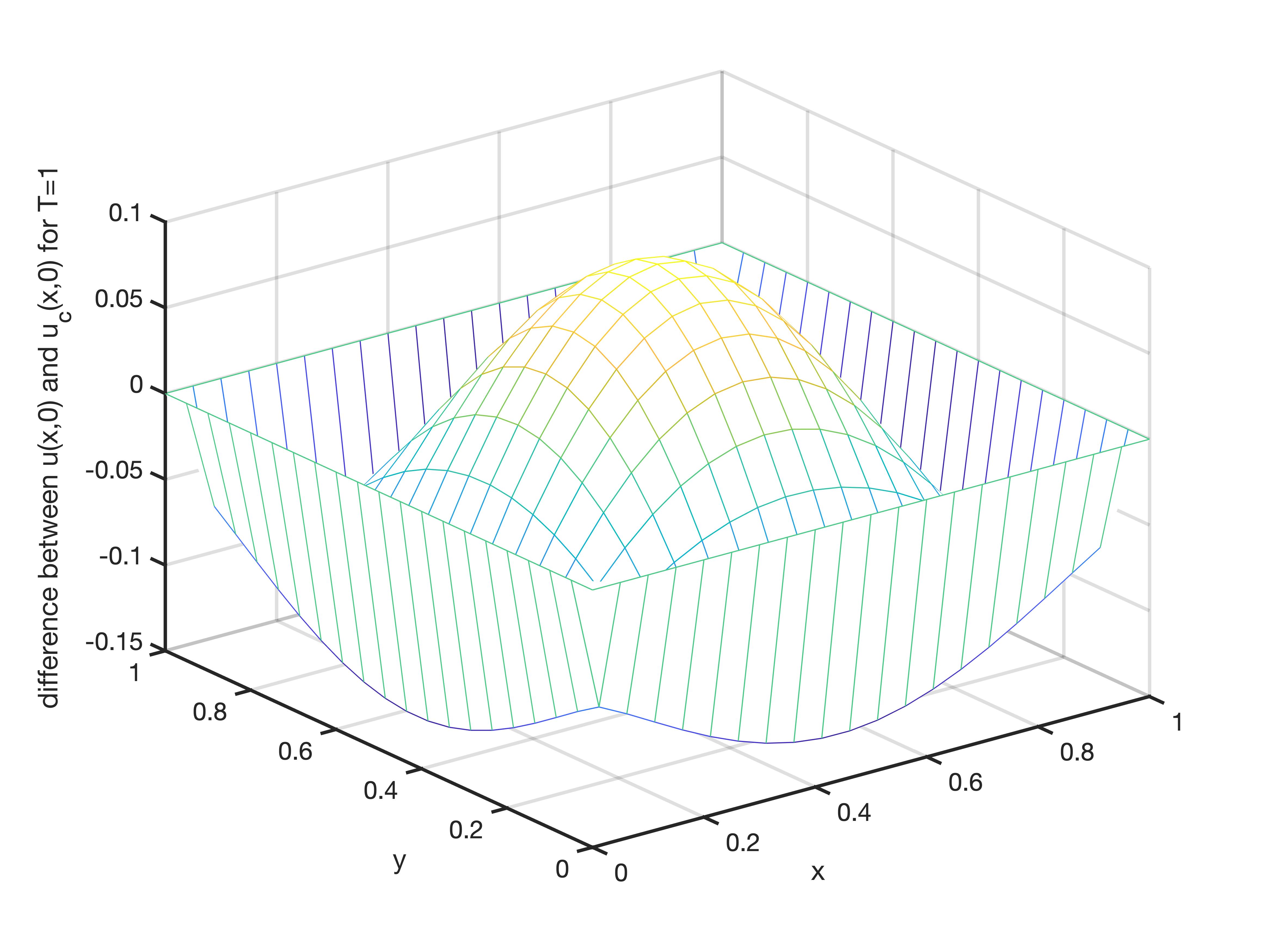}
}
\caption{(a) Exact solution. (b) Numerical solution with $\delta=0.02$. (c) Numerical solution with $\delta=0.05$. (d) Error with $\delta=0.02$ and $T=0.5$. (e) Error with $\delta=0.02$ and $T=0.75$. (f) Error with $\delta=0.02$ and $T=1$.}
\label{results}
\end{figure}

\section*{Acknowledgement}
The first author thanks the support of the NSFC (No. 12071061,11971093), the Applied Fundamental Research Program of Sichuan Province (No. 2020YJ0264), the Fundamental Research Funds for the Central Universities (No. ZYGX2019J094).

\bibliography{mybib}{}

\begin{thebibliography}{10}

\bibitem{barbu2}
Viorel Barbu and Michael R{\"o}ckner.
\newblock Backward uniqueness of stochastic parabolic like equations driven by
  gaussian multiplicative noise.
\newblock {\em Stochastic Processes and their Applications}, 126(7):2163--2179,
  2016.

\bibitem{BY2017book}
Mourad Bellassoued and Masahiro Yamamoto.
\newblock {\em Carleman estimates and applications to inverse problems for
  hyperbolic systems}.
\newblock Springer, 2017.

\bibitem{CWC2021IPI}
Shuli Chen, Zewen Wang, and Guolin Chen.
\newblock Cauchy problem of non-homogenous stochastic heat equation and
  application to inverse random source problem.
\newblock {\em Inverse Problems \& Imaging}, 15(4):619, 2021.

\bibitem{dPZ2014}
Giuseppe Da~Prato and Jerzy Zabczyk.
\newblock {\em Stochastic Equations in Infinite Dimensions}.
\newblock Encyclopedia of Mathematics and its Applications. Cambridge
  University Press, 2 edition, 2014.

\bibitem{D1972}
Donald~A Dawson.
\newblock Stochastic evolution equations.
\newblock {\em Mathematical biosciences}, 15(3-4):287--316, 1972.

\bibitem{DT2012}
Kai Du and Shanjian Tang.
\newblock Strong solution of backward stochastic partial differential equations
  in ${C}^2$ domains.
\newblock {\em Probability Theory and Related Fields}, 154(1-2):255--285, 2012.

\bibitem{DP2016SISC}
Thomas Dunst and Andreas Prohl.
\newblock The forward-backward stochastic heat equation: numerical analysis and
  simulation.
\newblock {\em SIAM Journal on Scientific Computing}, 38(5):A2725--A2755, 2016.

\bibitem{FLZ2019book}
Xiaoyu Fu, Qi~L{\"u}, and Xu~Zhang.
\newblock {\em Carleman Estimates for Second Order Partial Differential
  Operators and Applications: A Unified Approach}.
\newblock Springer Nature, 2019.

\bibitem{KT2012book}
Michael~V Klibanov and Alexander~A Timonov.
\newblock {\em Carleman estimates for coefficient inverse problems and
  numerical applications}.
\newblock de Gruyter, 2012.

\bibitem{KY2019IP}
Michael~V Klibanov and Anatoly~G Yagola.
\newblock Convergent numerical methods for parabolic equations with reversed
  time via a new carleman estimate.
\newblock {\em Inverse Problems}, 35(11):115012, 2019.

\bibitem{K2008}
Peter Kotelenez.
\newblock {\em Stochastic ordinary and stochastic partial differential
  equations: transition from microscopic to macroscopic equations}, volume~58.
\newblock Springer Science \& Business Media, 2007.

\bibitem{Krylov1994}
Nicolai~V Krylov.
\newblock Aw 2 n-theory of the dirichlet problem for spdes in general smooth
  domains.
\newblock {\em Probability Theory and Related Fields}, 98(3):389--421, 1994.

\bibitem{Li2013}
Hongheng Li and Qi~L{\"u}.
\newblock A quantitative boundary unique continuation for stochastic parabolic
  equations.
\newblock {\em Journal of Mathematical Analysis and Applications},
  402(2):518--526, 2013.

\bibitem{LZ2021S}
Q~L{\"u} and X~Zhang.
\newblock {\em Mathematical Control Theory for Stochastic Partial Differential
  Equations}.
\newblock Springer Nature, Switzerland AG, 2021.

\bibitem{L2012IP}
Qi~L{\"u}.
\newblock Carleman estimate for stochastic parabolic equations and inverse
  stochastic parabolic problems.
\newblock {\em Inverse Problems}, 28(4):045008, 2012.

\bibitem{LP2013}
Shuai Lu and Sergei~V Pereverzev.
\newblock {\em Regularization theory for ill-posed problems}.
\newblock de Gruyter, 2013.

\bibitem{Schuster2012}
Thomas Schuster, Barbara Kaltenbacher, Bernd Hofmann, and Kamil~S Kazimierski.
\newblock {\em Regularization methods in Banach spaces}.
\newblock de Gruyter, 2012.

\bibitem{WCW2020IP}
Bin Wu, Qun Chen, and Zewen Wang.
\newblock Carleman estimates for a stochastic degenerate parabolic equation and
  applications to null controllability and an inverse random source problem.
\newblock {\em Inverse Problems}, 36(7):075014, jul 2020.

\bibitem{Yamamoto}
Masahiro Yamamoto.
\newblock Carleman estimates for parabolic equations and applications.
\newblock {\em Inverse problems}, 25(12):123013, 2009.

\bibitem{Y2017IP}
Ganghua Yuan.
\newblock Conditional stability in determination of initial data for stochastic
  parabolic equations.
\newblock {\em Inverse Problems}, 33(3):035014, 2017.

\bibitem{Y2021JIIP}
Ganghua Yuan.
\newblock Inverse problems for stochastic parabolic equations with additive
  noise.
\newblock {\em Journal of Inverse and Ill-posed Problems}, 29(1):93--108, 2021.

\end{thebibliography}
\bibliographystyle{plain}

\end{document}